\documentclass[onefignum,onetabnum]{siamart190516}

\usepackage{amsmath,amssymb,amsxtra}
\usepackage{mathrsfs}
\usepackage{graphicx}
\usepackage{float,verbatim}
\usepackage{threeparttable,booktabs}
\usepackage{underscore}
\usepackage{algorithm}
\usepackage{subfigure}
\usepackage{extarrows}
\usepackage{multirow}
\usepackage{makecell}
\usepackage[noend]{algpseudocode}
 \usepackage{color}
\usepackage{hyperref}
\graphicspath{{./pic/}}
\newtheorem{remark}{Remark}[section]
\newtheorem{example}{Example}[section]

\allowdisplaybreaks
\input{siam.sty}

\title{Solving Poisson Problems in Polygonal Domains with Singularity Enriched Physics Informed Neural Networks\thanks{The work of B. Jin is supported by UK EPSRC grant EP/V026259/1, Hong Kong RGC General Research Fund (Project 14306423) and a start-up fund from The Chinese University of Hong Kong. The work of  Z. Zhou is supported by Hong Kong Research Grants Council (15303021) and an internal grant of Hong Kong Polytechnic University (Project ID: P0038888, Work Programme: 1-ZVX3).}}

\author{Tianhao Hu\thanks{Department of Mathematics, The Chinese University of Hong Kong, Shatin, New Territories, Hong Kong, P.R. China (\texttt{1155202953@link.cuhk.edu.hk, bangti.jin@gmail.com, b.jin@cuhk.edu.hk}).}
\and Bangti Jin\footnotemark[2]
\and Zhi Zhou\footnotemark[3]\thanks{Department of Applied Mathematics,
The Hong Kong Polytechnic University, Kowloon, Hong Kong, P.R. China (\texttt{zhizhou@polyu.edu.hk})}}
\date{\today}

\begin{document}

\maketitle

\begin{abstract}
Physics-Informed Neural Networks (PINNs) are a powerful class of numerical solvers for partial differential equations, employing deep neural networks with successful applications across a diverse set of problems. However, their effectiveness is somewhat diminished when addressing issues involving singularities, such as point sources or geometric irregularities, where the approximations they provide often suffer from reduced accuracy due to the limited regularity of the exact solution. In this work, we investigate PINNs for solving Poisson equations in polygonal domains with geometric singularities and mixed boundary conditions. We propose a novel singularity enriched PINN (SEPINN), by explicitly incorporating the singularity behavior of the analytic solution, e.g., corner singularity, mixed boundary condition and edge singularities, into the ansatz space, and present a convergence analysis of the scheme. We present extensive numerical simulations in two and three-dimensions to illustrate the efficiency of the method, and also a comparative study with several existing neural network based approaches.
\end{abstract}


\begin{AMS}
{65N12}
\end{AMS}

\pagestyle{myheadings}
\thispagestyle{plain}

\section{Introduction}

Partial differential equations (PDEs) represent a broad class of mathematical models that occupy a vital role in physics, science and engineering. Many traditional PDE solvers have been developed, e.g., finite difference method, finite element method and finite volume method. They have been maturely developed over past decades, and efficient implementations and mathematical guarantees are also available. In the last few years, motivated by the great successes in diverse areas (computer vision, speech recognition and natural language processing etc), neural solvers for PDEs using deep neural networks (DNNs) have received much attention \cite{lagaris1998artificial}. The list of neural solvers includes physics informed neural networks (PINNs) \cite{RAISSI2019686}, deep Ritz method (DRM) \cite{yu2018deep}, deep Galerkin method \cite{sirignano2018dgm}, weak adversarial network \cite{Zang:2020} and deep least-squares method \cite{CaiChenLiu:2021}, to name a few. Compared with traditional methods, neural PDE solvers have shown very promising results in several direct and inverse problems \cite{Karniadakis:2021nature,EHanJentzen:2022,CuomoSchiano:2022}.

In these neural solvers, one employs DNNs as ansatz functions to approximate the solution to the PDE either in strong, weak or Ritz forms. Existing approximation theory of DNNs \cite{GuhringRaslan:2021} indicates that the accuracy of DNN approximations depends crucially on the Sobolev regularity of the solution (also suitable stability of the mathematical formulation). Thus, these methods might be ineffective or even fail completely when applied to problems with irregular solutions \cite{WangPerdikaris:2022jcp,Krishnapriyan:2021}, e.g., convection-dominated problems, transport problem, high-frequency wave propagation, problems with geometric singularities (cracks / corner singularity) and singular sources. All of these settings lead to either strong directional behavior, solution singularities or highly oscillatory behavior, which are challenging for standard DNNs to approximate effectively.

Thus, there is an imperative need to develop neural solvers for PDEs with nonsmooth solutions. Several recent efforts have been devoted to addressing the issue, including self-adaptive PINN (SAPINN) \cite{huang2021solving}, failure-informed PINN (FIPINN) \cite{GaoTangYanZhou:2023, GaoYanZhou:2022} and singularity splitting DRM (SSDRM) \cite{HuJinZhou:2022}. SAPINN extends the PINN by splitting out the regions with singularities and then setting different weights to compensate the effect of singular regions. FIPINN \cite{GaoYanZhou:2022} is inspired by the classical adaptive FEM, using the PDE residual as the indicator to aid judicious selection of sampling points for training. SSDRM \cite{HuJinZhou:2022} exploits analytic insights into the exact solution, by approximating only the regular part using DNNs whereas extracting the singular part explicitly. These methods have shown remarkable performance for problems with significant singularities. One prime example is point sources, whose solutions involve localized singularities that can be extracted using fundamental solutions. See Section \ref{ssec:existing} for further discussions about these methods.

In this work, we continue this line of research for Poisson problems on polygonal domains, which involve geometric singularities, including corners and mixed boundary condition in the two-dimensional (2D) case, and edges in the three-dimensional (3D) case. This represents an important setting in practical applications that has received enormous attention; see \cite{grisvard2011elliptic,KozlovMazya:1997,KozlovMazya:2001,MazyaRossmann:2010} for the solution  theory. We shall develop a class of effective neural solvers for Poisson problems with geometric singularities based on the idea of singularity enrichment, building on known analytic insights of the problems, and term the proposed method singularity enriched PINN (SEPINN).

\subsection{Problem setting}
First, we state the mathematical formulation of the problem. Let $\Omega\in\mathbb{R}^d$ ($d=2,3$) be an open, bounded polygonal domain with a boundary $\partial\Omega$, and $ \Gamma_D $ and $ \Gamma_N $ be a partition of $\partial \Omega $ such
that $ \Gamma_D\cup\Gamma_N=\partial\Omega $ and $ \Gamma_D\cap\Gamma_N=\emptyset $,  with a nonempty $ \Gamma_D $ (i.e.,  Lebesgue measure $|\Gamma_D|\neq 0 $). Let $ n $ denote the unit outward normal vector to $\partial\Omega$, and $\partial_nu$ denote taking the outward normal derivative.  Given a source $f\in L^2(\Omega)$, consider
the following Poisson problem
\begin{align}\label{problem}
	\left\{\begin{aligned}
		-\Delta u &= f,&& \mbox{in }\Omega,\\
		u&=0,&& \mbox{on }\Gamma_D,\\
		\partial_n u&=0,&& \mbox{on }\Gamma_N.
	\end{aligned}\right.
\end{align}
We focus on the zero boundary conditions, and nonzero ones can be transformed to \eqref{problem} using the trace theorem.
Due to the existence of corners, cracks or edges in $\Omega$, the solution $u$ of problem \eqref{problem} typically exhibits singularities, even if $f$ is smooth. The presence of singularities in the solution $u$ severely deteriorates the accuracy of standard numerical methods for constructing approximations, including neural solvers, and more specialized techniques are needed in order to achieve high efficiency. Next we briefly review existing techniques for resolving the singularities.

In the 2D case, there are several classes of traditional numerical solvers based on FEM, including singularity representation based approaches \cite{Fix:1973,StephanWhiteman:1988,cai2001finite}, mesh grading \cite{Raugel:1978,ApelWhiteman:1996,ApelNicaise:1998}, generalized FEM \cite{Fries:2010} and adaptive FEM \cite{GaspozMorin:2009} etc. These methods require different amount of  knowledge about
the analytic solution. The methods in the first class exploits a singular representation of the solution $u$ as a linear combination of singular and regular parts \cite{grisvard2011elliptic,DaugeM}, and can be further divided into four groups.
\begin{itemize}
    \item[(i)] The singular function method augments singular functions to both trial and test spaces \cite{Fix:1973,StephanWhiteman:1988}. However, the convergence of the coefficients (a.k.a. stress intensity factors) sometimes is poor \cite{DestuynderDjaoua:1980}, which may lead to low accuracy. 
    \item[(ii)] The dual singular function method \cite{BlumDobrowolski:1982} employs the dual singular function to extract the coefficients as a postprocessing strategy of FEM, which can also achieve the theoretical rate in practical computation.
    \item[(iii)] The singularity splitting method  \cite{cai2001finite} splits the singular part from the solution $u$ and approximates the smooth part with the Galerkin FEM, and enjoys $H^1(\Omega)$ and $L^2(\Omega)$ error estimates. It can improve the accuracy of the approximation, and the stress intensity factor can also be obtained from the extraction formula, cf. \eqref{eqn:extraction}.
    \item[(iv)] The singular complement method \cite{AssousCiarlet:2000} is based on an orthogonal decomposition of the solution $u$ into a singular part and a regular part, by augmenting the FEM trial space with specially designed singular functions.
\end{itemize}

For 3D problems with edges, the singular functions belong to an infinite dimensional space and their
coefficients are functions defined along edges \cite{grisvard2011elliptic}. Thus, their computation involves approximating
functions defined along edges, and there are relatively few  numerical methods, and numerical
investigations are strikingly lacking. The methods developed for the 2D case do not extend directly to the 3D case. 
In fact, in several existing studies, numerical algorithms and error analysis have been provided, but the methods are nontrivial
to implement  \cite{nkemzi2021singular,StephanWhiteman:1988}. For example, the approach in \cite{nkemzi2021singular} requires evaluating a few dozens of highly singular
integrals at each step, which may lead to serious numerical issues.

\subsection{Our contributions}
In this work, by building analytic knowledge of the problem into numerical schemes, we construct a novel numerical method using PINN to solve Poisson  problems with geometric singularities. The key analytic insight is that the solution $u$  has a singular function representation as a linear combination of a singular function $S$ and a regular part $w$ \cite{DaugeM,1999Singularities,grisvard2011elliptic,1990Singularities}:
$u=S+w$, with $w\in H^2(\Omega)$. The singular function $S$ is determined by the domain $\Omega$, truncation functions and their coefficients. Using this fact, we develop, analyze and test SEPINN, and make the following contributions:
\begin{itemize}
  \item[(i)] develop a novel class of SEPINNs for corner singularities, mixed boundary conditions and edge singularity, for the Poisson problem. 
  \item[(ii)] provide error bounds for the SEPINN approximation.
  \item[(iii)] present numerical experiments for multiple scenarios, including 3D problems with edge singularities and the eigenvalue problem on an L-shaped domain, to illustrate the flexibility and accuracy of SEPINN. We also include a comparative study with existing approaches.
\end{itemize}
To the best of our knowledge, it is the first work systematically exploring the use of singularity enrichment in a neural PDE solver.

The rest of the paper is organized as follows. In Section \ref{sec:prelim} we recall preliminaries on DNNs and its use in PINNs. Then we develop the singularity enriched PINN in Section \ref{sec:SEPINN} for 2D case (corner singularity and mixed boundary conditions) and 3D case (edge singularity) separately. 
In Section \ref{sec:error}, we discuss the convergence analysis of SEPINN. In Section \ref{sec:experiment}, we present  numerical experiments to illustrate the performance of SEPINN, including a comparative study with PINN and its variants, and give further discussions in Section \ref{sec:concl}.
\section{Preliminaries}\label{sec:prelim}

\subsection{Deep neural networks}
We employ standard fully connected feedforward DNNs, i.e., functions $ f_\theta: \mathbb{R}^d\rightarrow \mathbb{R}$, with the DNN parameters $\theta \in\mathbb{R}^{N_\theta}$ ($N_\theta$ is the dimensionality of the DNN parameters). Given a sequence of integers $\{n_\ell\}_{\ell=0}^L$, with $n_0=d$ and $n_L=1$, $f_\theta({x})$ is defined recursively by:
$f^{(0)}={x}$,
$f^{(\ell)}= \varrho({A}_\ell f^{({\ell-1})}+{b}_\ell)$, $ \ell=1,2,\cdots,L-1$,
$f_\theta({x}):=f^{(L)}(x)={A}_Lf^{(L-1)}+{b}_L$,
where ${A}_\ell\in\mathbb{R}^{n_\ell\times n_{\ell-1}}$ and ${b}_\ell\in\mathbb{R}^{n_\ell},\, \ell=1,2,\cdots,L$, are the weight matrix and bias vector at the $\ell$th layer. The nonlinear activation function $ \varrho:\mathbb{R}\to \mathbb{R}$ is applied componentwise to a vector. The integer $L$ is called the depth, and $W:=\max\{n_\ell,\ell=0,1,\cdots,L\}$ the width of the DNN. The set of parameters $\{{A}_\ell,\,{b}_\ell\}_{\ell=1}^L$ of the DNN is trainable and stacked into a big vector $\theta$.  $f^{(0)}$ is called the input layer, $f^{(\ell)}$, $\ell=1,2,\cdots,L-1$, are called the hidden layer and $f_\theta({x})$ is the output layer.

There are many possible choices of $ \varrho$. The most frequently used one in computer vision is rectified linear unit (ReLU), $ \varrho(x)=\max(x,0)$. However, it is not smooth enough for PINN, since PINN requires thrice differentiability of $ \varrho$: two spatial derivatives in the loss, and another derivative in the DNN parameters $\theta$ (for the optimizer). For neural PDE solvers, hyperbolic tangent $ \varrho(x)=\frac{{\rm e}^x-{\rm e}^{-x}}{{\rm e}^x+{\rm e}^{-x}}$ and logistic / sigmoid $  \varrho(x)=\frac{1}{1+{\rm e}^{-x}} $ are often used \cite{RAISSI2019686,CuomoSchiano:2022}. We employ the hyperbolic tangent.
We denote the collection of DNN functions of depth $L$, with $N_\theta$ nonzero parameters, and each the parameter bounded by $B_\theta$, with the activation function $ \varrho$ by $\mathcal{N}_ \varrho(L,N_\theta, B_\theta)$, i.e., $\mathcal{N}_ \varrho(L,N_\theta,B_\theta) = \{w_\theta: w_\theta ~ \mbox{has a depth } L,\, |\theta|_0\leq N_\theta, \, |\theta|_{\ell^\infty}\leq B_\theta\}$,
where $|\cdot|_{\ell^0}$ and $|\cdot|_{\ell^\infty}$ denote the number of nonzero entries in and the maximum norm of a vector, respectively. We also use the notation $\mathcal{A}$ to denote this collection of functions.

\subsection{Physics informed neural networks}\label{ssec:pinn}
Physics informed neural networks (PINNs) \cite{RAISSI2019686} represent one popular neural solver based on the principle of PDE residual minimization. For problem \eqref{problem}, the continuous loss $\mathcal{L}(u)$ is given by
\begin{equation}
    \mathcal{L}_{\boldsymbol\sigma}(u) = \|\Delta u + f\|_{L^2(\Omega)}^2 + \sigma_d \|u\|_{L^2(\Gamma_D)}^2 + \sigma_n \|\partial_nu\|_{L^2(\Gamma_N)}^2,
\end{equation}
where the tunable penalty weights $\sigma_d,\sigma_n>0$ are to approximately enforce the boundary conditions, and $\boldsymbol{\sigma}=(\sigma_d,\sigma_n)$. We approximate the solution $u$ by an element $u_\theta \in \mathcal{A}$, and then discretize relevant integrals using quadrature, e.g., Monte Carlo method. Let $U(D)$ be the uniform distribution over a set $D$, and let $|D|$ denote its Lebesgue measure. Then we can rewrite the loss $\mathcal{L}(u_\theta)$ as
\begin{align*}
    \mathcal{L}_{\boldsymbol\sigma}(u_\theta) =& |\Omega|\mathbb{E}_{U(\Omega)}[(\Delta u_\theta(X) + f(X))^2] + \sigma_d |\Gamma_D|\mathbb{E}_{U(\Gamma_D)}[(u_\theta(Y))^2] \\
     &+ \sigma_n|\Gamma_N|\mathbb{E}_{U(\Gamma_N)} [(\partial_nu_\theta(Z))^2],
\end{align*}
where $\mathbb{E}_\nu$ takes expectation with respect to a distribution $\nu$. Let the sampling points $\{X_i\}_{i=1}^{N_r}$, $\{Y_j\}_{j=1}^{N_d}$ and $\{Z_k\}_{k=1}^{N_n}$ be identically and independently distributed (i.i.d.), uniformly on $\Omega$, $\Gamma_D$ and $\Gamma_N$, respectively, i.e., $\{X_i\}_{i=1}^{N_r}\sim U(\Omega)$, $\{Y_j\}_{j=1}^{N_d}\sim U(\Gamma_D)$ and $\{Z_k\}_{k=1}^{N_n}\sim U(\Gamma_N)$. Then the empirical loss $\widehat{\mathcal{L}}_{\boldsymbol\sigma}(u_\theta)$ is given by
\begin{equation*}
    \widehat{\mathcal{L}}_{\boldsymbol\sigma}(u_\theta) = \frac{|\Omega|}{N_r}\sum_{i=1}^{N_r}(\Delta u_\theta(X_i) + f(X_i))^2+ \frac{\sigma_d |\Gamma_D|}{N_{d}}\sum_{j=1}^{N_d}(u_\theta(Y_j))^2 + \frac{\sigma_n|\Gamma_N|}{N_n}\sum_{k=1}^{N_n} (\partial_nu_\theta(Z_k))^2.
\end{equation*}
Note that the resulting optimization problem $\widehat{\mathcal{L}}_{\boldsymbol\sigma}(u_\theta)$ over $\mathcal{A}$ is well posed due to the box constraint on the DNN parameters $\theta$, i.e., $|\theta|_{\ell^\infty}\leq R$ for suitable $R$, which induces a compact set in $\mathbb{R}^{N_\theta}$. Meanwhile the empirical loss $\widehat{\mathcal{L}}_{\boldsymbol\sigma}(u_\theta)$ is continuous in $\theta$, when $ \varrho$ is smooth. In the absence of the box constraint, the optimization problem might not have a finite minimizer.

The loss $\widehat{\mathcal{L}}_{\boldsymbol\sigma}(u_\theta)$ is minimized with respect to the DNN  parameters $\theta$. This is often achieved by gradient type algorithms, e.g.,  Adam \cite{KingmaBa:2015} or limited memory BFGS \cite{ByrdLu:1995}, which returns an approximate minimizer $\theta^*$. The DNN approximation to the PDE solution $u$ is given by $u_{\theta^*}$. Note that the major computational effort, e.g., gradient of $\widehat{\mathcal{L}}_{\boldsymbol\sigma}(u_\theta)$ with respect to the DNN parameters $\theta$  and the DNN $u_\theta(x)$ with respect to the input $x$ can both be computed efficiently via automatic differentiation \cite{Baydin:2018}, which is available in many software platforms, e.g., PyTorch or Tensorflow. Thus, the method is very flexible and easy to implement, and applicable to a wide range of direct and inverse problems for PDEs  \cite{Karniadakis:2021nature}.

The population loss $\mathcal{L}_{\boldsymbol\sigma}(u_\theta) $ and empirical loss $ \widehat{\mathcal{L}}_{\boldsymbol\sigma}(u_\theta)$ have different minimizers,
due to the presence of quadrature errors. The analysis of these errors is known as generalization error analysis in statistical learning theory \cite{AnthonyBartlett:1999}. The theoretical analysis of PINNs has been investigated  in several works under different settings \cite{ShinDarbon:2020,JiaoLai:2022cicp,MishraMolinaro:2023,LuChenLu:2021,DeRyckMishra:2023acm,DeRyckJagtapMishra:2023ima}. Important issues addressed in these works include the existence of a neural network such that the population / empirical loss is small, connection between the error of the approximate solution with the population loss (via suitable stability of the formulation), and the gap between the population loss and training loss. These roughly correspond to approximation theory of DNNs in Sobolev spaces, coercivity of the formulation, and (Monte Carlo) quadrature error. These mathematical theories require that the solutions to the problems be smooth, e.g., $C^2(\overline{\Omega})$ in order to achieve consistency \cite{ShinDarbon:2020} and even stronger regularity for convergence rates \cite{JiaoLai:2022cicp}. Such conditions unfortunately cannot be met for problem \eqref{problem}, due to the inherently limited solution regularity. Thus, it is not \textit{a priori} clear that one can successfully apply PINNs to problem \eqref{problem}. This is also confirmed by the numerical experiments in Section \ref{sec:experiment}. In practice, despite the ease of implementation, the standard PINN may fail to converge \cite{Krishnapriyan:2021,WangPerdikaris:2022jcp}. Indeed, there are several failure modes, including the curse of dimensionality, low regularity and discontinuities etc. The concerned class of problems involves solution singularity, which represents one typical failure mode of PINN. In this work, we shall develop an effective strategy to overcome the challenge.

\subsection{Two existing methods}\label{ssec:existing} Now we describe two PINN approaches to address the singularity, i.e., SAPINN
 and FIPINN, for the Dirichlet problem of the Laplace equation with a singular source $f$ with singularities located within a subdomain $\Omega_s\subset \Omega$.

In SAPINN \cite{huang2021solving}, one treats the PDE residuals in the regions $\Omega_s$ and $\Omega\setminus\Omega_s$ separately and further splits the loss into
\begin{equation*}
		L(u_\theta)=\lambda_sL_{s}(u_\theta)+\lambda_rL_{r}(u_\theta)+\lambda_bL_{b}(u_\theta),
\end{equation*}
with $ L_{s}(u_\theta)=\|\Delta u_\theta+f\|_{L^2(\Omega_s)}^2$,
$L_{r}(u_\theta)=\|\Delta  u_\theta+f\|_{L^2(\Omega\backslash \Omega_s)}^2$ and $
L_{b}(u_\theta)=\|u_\theta\|_{L^2(\partial\Omega)}^2$,
and $\lambda_i$, $i\in \{r,s,b\}$, being positive weights. The partition of $\Omega$ into $\Omega_s$ and $\Omega\subset \Omega_s$ and using different penalty parameters allow better compensation for the singular behavior of $u$ in the subdomain $\Omega_s$, in a manner similar to weighted Sobolev spaces. The experimental results in \cite{huang2021solving} indicate that the method can improve the accuracy of the approximation, but the increase in the number of hyperparameters (i.e., $\lambda_i$s) also complicates the optimization process, due to the need for self adaptation during training, which in turn makes the overall training lengthier.

FIPINN \cite{GaoYanZhou:2022} employs adaptive sampling. Let $\mathcal{Q}(x):\Omega\to [0,\infty)$ be a performance function, e.g., the PDE residual $r(x; \theta)=(\Delta u_\theta(x)+f(x))^2$ and its spatial gradient $\nabla_x r(x;\theta)$, and  $ \Omega_\mathcal{F} := \{x : \mathcal{Q}(x) > \epsilon_r\}$ be the failure region (with tolerance $\epsilon_r$). Then over $\Omega_\mathcal{F}$, one can define the failure probability
$P_\mathcal{F} = \int_{\Omega_\mathcal{F}}
\omega(x) {\rm d}x$, where $\omega(x)$ is the prior distribution of the sampling point $x$.
The failure probability $P_\mathcal{F}$ may serve as an error indicator for training set refinement, similar to the classical adaptive FEM. Using $P_\mathcal{F}$, new candidate collocation points can be generated to form the loss in the hope of enhancing PINN's performance after retraining (simultaneously for estimating $P_\mathcal{F})$. In \cite{GaoYanZhou:2022}, a
truncated Gaussian is taken for importance sampling in order to approximatively determine $P_\mathcal{F}$ and to produce new training points. These steps are looped until a suitable stopping criterion is met.

{Note that to compensate the impact of solution singularity, SAPINN and FIPINN follow different strategies: SAPINN adaptively updates the weights assigned to the losses in the subdomains, whereas FIPINN employs the residual as an a posteriori error estimator to adaptively generate sampling points in the empirical distribution.}

\section{Singularity enriched PINN}\label{sec:SEPINN}
Now we develop a class of singularity enriched PINN (SEPINN) for solving Poisson problems with geometric singularities, including 2D problems with mixed boundary conditions or on polygonal domains, and 3D problems with edge singularity. The key is the singular function representation. We discuss the 2D case in Section \ref{2d domain}, and the more involved 3D case in  Section \ref{3d domain}. The approach applies also to the modified Helmholtz equation; which we describe in the supplementary material.

\subsection{Two-dimensional problem}\label{2d domain}
\subsubsection{Singular function representation}\label{2d:singular repre}
First we develop a singular function representation in the 2D case.
We determine the analytic structure of the solution $u$ of  problem \eqref{problem} using the Fourier method. Consider a vertex $\boldsymbol{v}_j$ of the polygonal domain $\Omega$ with an interior angle $\omega_j$. We denote by $(r_j,\theta_j)$ the local polar coordinate of the vertex $\boldsymbol{v}_j$ so that the interior angle $\omega_j$ is spanned counterclockwise by two rays $ \theta_j = 0 $ and $ \theta_j = \omega_j $. Then consider the local behavior of $u$ near the vertex $\boldsymbol{v}_j$, or in  the sector
$G_j=\{(r_j,\theta_j): 0<r_j<R_j, 0<\theta_j<\omega_j\}.$
Let the edge overlapping with $ \theta_j = 0 $ be $\Gamma_{j_1}$ and with $ \theta_j = \omega_j $ be $\Gamma_{j_2}$, cf. Fig. \ref{fig:sketch}(a) for the sketch. We employ the system of orthogonal and complete set of basis functions $\{\phi_{j,k}\}_{k=1}^\infty$ in $ L^2(0,\omega_j) $ in Table \ref{table 2.1}, where $\lambda_{j,k}$ are the eigenvalues. More precisely, the pairs $(\phi_{j,k},\lambda_{j,k})$ $k=1,\ldots,\infty$, are associated with the following Sturm-Liouville problem on the interval $(0,\omega_j)$:
\begin{equation*}
    -\phi''_{j,k} = \lambda_{j,k} \phi_{j,k}\quad \mbox{in } (0,\omega_j),
\end{equation*}
with suitable zero Dirichlet / Neumann boundary conditions, according to the boundary conditions (b.c.) on the edges $\Gamma_{j_1}$ and $\Gamma_{j_2}$ \cite{2012Flux}.
In the table, we have dropped the subscript $j$ for notational simplicity. We denote the representation of $u$ in the local polar coordinate by $\widetilde u(r_j,\theta_j)$, i.e., $\tilde u(r_j,\theta_j)=u(r_j\cos\theta_j,r_j\sin \theta_j)$.

\begin{figure}[hbt!]
	\centering\setlength{\tabcolsep}{0pt}
	\begin{tabular}{cc}
		\includegraphics[height=4.8cm]  {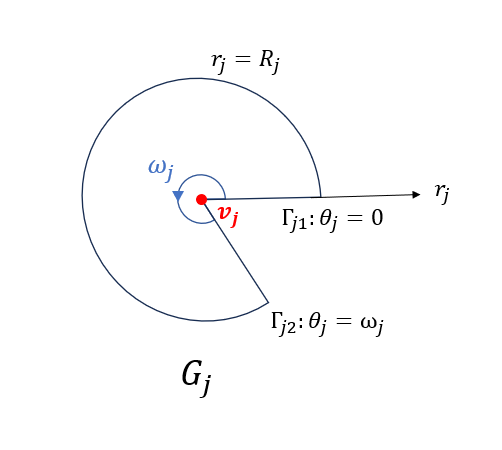} &  \includegraphics[height=4.8cm]  {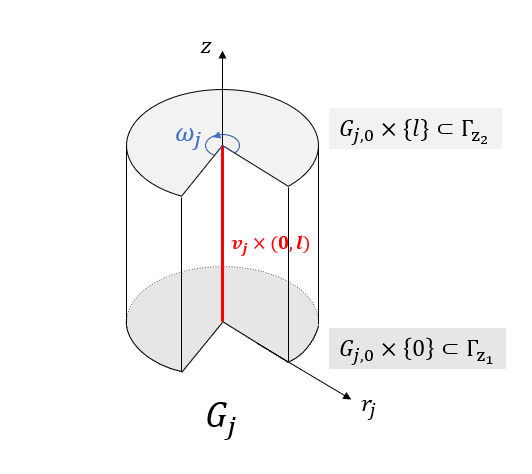}\\
		(a) 2D & (b) 3D
	\end{tabular}
	\caption{\label{fig:sketch} The sketch of the local domain $G_j$ in two- and three-dimensions. In the 3D case,$G_{j,0}$ is defined by $G_{j,0}=\{(x_j,y_j)\in\Omega_0:0<r_j<R_j,0<\theta_j<\omega_j\}$.}
\end{figure}

\begin{table}[hbt]
\centering
\begin{threeparttable}  \caption{Orthogonal basis functions $\phi_{j,k}$ (dropping subscript $j$), $k\in\mathbb{N}$, associated with the vertex $\boldsymbol{v}_j$ in a local polar coordinates, depending on the boundary conditions on $\Gamma_{j_1}$ and $\Gamma_{j_2}$.}\label{table 2.1}
		\begin{tabular}{c|c|c}
			\toprule
		{$\Gamma_{j_2}\backslash\Gamma_{j_1}$} & Dirichlet & Neumann \\
			\midrule
			Dirichlet & $\phi_{k}(\theta)=\sin\lambda_{k}\theta$, $\lambda_{k}=\frac{k\pi}{\omega}$&
			$\phi_{k}(\theta)=\sin\lambda_{k}\theta$, $\lambda_{k}=(k-\frac{1}{2})\frac{\pi}{\omega}$ \\
			\midrule
			Neumann & $\phi_{k}(\theta)=\cos\lambda_{k}\theta$, $\lambda_{k}=(k-\frac{1}{2})\frac{\pi}{\omega}$
			&
			$\phi_{k}(\theta)=\cos\lambda_{k}\theta$, $\lambda_{k}=(k-1)\frac{\pi}{\omega}$ \\
			\bottomrule
		\end{tabular}
	\end{threeparttable}
\end{table}

To study the behavior of the solution $u$ in the sector $G_j$, we assume $\widetilde{u}(R_j,\theta_j)=0$ and $|\widetilde{u}(0,\theta_j)|<\infty$. Since $u$ and $f$ belong to $L^2(G_j)$, they can be represented in Fourier series with respect to $\{\phi_{j,k}\}_{k=1}^\infty$:
\begin{align}
u(x,y)&=\widetilde{u}(r_j,\theta_j)=\sum_{k=1}^{\infty}u_k(r_j)\phi_{j,k}(\theta_j),\quad\mbox{in }G_j,\label{fourier u}\\
f(x,y)&=\widetilde{f}(r_j,\theta_j)=\sum_{k=1}^{\infty}f_k(r_j)\phi_{j,k}(\theta_j),\quad \mbox{in }G_j,\label{fourier f}
\end{align}
with $u_k(r_j)$ and $f_k(r_j)$ given respectively by
$u_k(r_j)=\frac{2}{\omega_j}\int_{0}^{\omega_j}\widetilde{u}(r_j,\theta_j)\phi_{j,k}(\theta_j){\rm d}\theta_j$ and  $f_k(r_j)=\frac{2}{\omega_j}\int_{0}^{\omega_j}\widetilde{f}(r_j,\theta_j)\phi_{j,k}(\theta_j){\rm d}\theta_j.$
Substituting \eqref{fourier u} and \eqref{fourier f} into \eqref{problem} gives the following two-point boundary value problem for each vertex $\boldsymbol{v}_j$:
\begin{align*}
  \left\{\begin{aligned}
     -u_{j,k}'' - r^{-1} u_{j,k}' +\lambda_{j,k}^2r^{-2}u_{j,k} &= f_k, \quad 0<r<R_j,\\
     |u_{j,k}(r)|_{r=0}<\infty, \quad u_{j,k}(r)|_{r=R_j} &= 0,
     \end{aligned}\right.
\end{align*}
Solving the ODE yields directly the closed-form expression
$$u_k(r_j)=c(r_j)r_j^{\lambda_{j,k}} +\frac{r_j^{-\lambda_{j,k}}}{2\lambda_{j,k}}\int_{0}^{r_j}f_k(\tau)\tau^{1+\lambda_{j,k}}{\rm d}\tau,$$
with the factor $c(r_j)$ given by \cite[(2.10)]{2012Flux}
$$c(r_j)=\frac{1}{2\lambda_{j,k}}\int_{r_j}^{R_j}f_k(\tau)\tau^{1-\lambda_{j,k}}{\rm d}\tau - \frac{1}{2\lambda_{j,k} R_j^{2\lambda_{j,k}}}\int_{0}^{R_j}f_k(\tau)\tau^{1+\lambda_{j,k}}{\rm d}\tau.$$
Since the factor $c(r_j)$ in front of $r_j^{\lambda_{j,k}}$ is generally nonzero, and $r_j^{\lambda_{j,k}}\notin H^2(G_j)$ for $\lambda_{j,k}<1$, there exist singular terms in the representation \eqref{fourier u}. This shows the limited regularity of the solution $u$, which is the culprit of low efficiency of standard numerical schemes.

Next we list all singularity points and corresponding singular functions. Let $\boldsymbol{v}_j$, $j=1,2,\cdots,M$, be the vertices of $\Omega$ whose interior angles $\omega_j,j=1,2,\cdots,M$, satisfy
\begin{align}\label{condition}
	\left\{\begin{aligned}
		\pi<\omega_j<2\pi,&& \mbox{b.c. doesn't change its type},\\
		\pi/2<\omega_j<2\pi,&& \mbox{b.c. changes its type}.
	\end{aligned}\right.
\end{align}
Table \ref{tab:sing-2d} gives the index set $\mathbb{I}_j$ and the associated singularity functions \cite[p. 2637]{cai2006finite}. Note that $s_{j,1}\in H^{1+\frac{\pi}{\omega}-\epsilon}(\Omega)$, $s_{j,\frac{1}{2}}\in H^{1+\frac{\pi}{2\omega}-\epsilon}(\Omega)$, and $s_{j,\frac{3}{2}}\in H^{1+\frac{3\pi}{2\omega}-\epsilon}(\Omega)$ for small $\epsilon>0$. Upon letting
\begin{align*}
	\omega^*=\max_{1\leq j\leq M}{\hat{\omega_j}},\quad\mbox{with }\hat{\omega_j}=\left\{\begin{aligned}
		\omega_j,&& \mbox{b.c. doesn't change its type at }\boldsymbol{v}_j,\\
		2\omega_i,&& \mbox{b.c. changes its type at }\boldsymbol{v}_j,
	\end{aligned}\right.
\end{align*}
the solution $u$ belongs to $H^{1+\frac{\pi}{\omega^*}-\epsilon}(\Omega)$, just falling short of $H^{1+\frac{\pi}{\omega^*}}(\Omega)$. If $\omega^*>\pi $, $u$ fails to belong to $H^2(\Omega)$. Hence, it is imperative to develop techniques to resolve such singularities.

\begin{table}[hbt]
\centering\setlength{\tabcolsep}{5pt}
\begin{threeparttable}
\caption{Singularity functions $s_{j,i}$ (suppressing the subscript $j$), depending on the boundary condition. The tuple $(r_j,\theta_j)$ refers to the local polar coordinate of the vertex $\boldsymbol{v}_j$, and $\mathbb{I}_j$ is an index set for leading singularities associated with $\boldsymbol{v}_j$.}\label{tab:sing-2d}
		\begin{tabular}{c|c|c}
			\toprule
      {$\Gamma_{j2}\backslash\Gamma_{j1}$} & Dirichlet & Neumann \\
			\midrule
			\multirow{2}{*}{Dirichlet} & \multirow{2}{*}{$ s_{1}(r,\theta_j)=r^\frac{\pi}{\omega}\sin\frac{\pi\theta}{\omega} $,} & \makecell{$ s_{\frac{1}{2}}(r,\theta)=r^\frac{\pi}{2\omega}\cos\frac{\pi\theta}{2\omega} $,\\[5pt]$\mathbb{I}=\left\{\frac{1}{2}\right\}$, if $\frac{\pi}{2}<\omega_j\leq\frac{3\pi}{2}$}\\
			\cline{3-3}
			 & $\mathbb{I}=\{1\}$&
			\makecell{\\$ s_{\frac{1}{2}}(r,\theta)=r^\frac{\pi}{2\omega}\cos\frac{\pi\theta}{2\omega} $ and\\ $ s_{\frac{3}{2}}(r,\theta)=r^\frac{3\pi}{2\omega}\cos\frac{3\pi\theta}{2\omega} $\\[5pt]$\mathbb{I}=\left\{\frac{1}{2},\frac{3}{2}\right\}$, if $\frac{3\pi}{2}<\omega\leq2\pi$} \\
			\midrule
			\multirow{2}{*}{Neumann} &  \makecell{$s_{\frac{1}{2}}(r,\theta)=r^\frac{\pi}{2\omega}\sin\frac{\pi\theta}{2\omega} $,\\[5pt]$\mathbb{I}=\left\{\frac{1}{2}\right\}$, if $\frac{\pi}{2}<\omega\leq\frac{3\pi}{2}$} & \multirow{2}{*}{$ s_{1}(r,\theta_j)=r^\frac{\pi}{\omega}\cos\frac{\pi\theta}{\omega} $,}\\
			\cline{2-2}
			&
			\makecell{$ s_{\frac{1}{2}}(r,\theta)=r^{\frac{\pi}{2\omega}}\sin\tfrac{\pi\theta}{2\omega} $ and\\ $ s_{\frac{3}{2}}(r,\theta)=r^{\frac{3\pi}{2\omega}}\sin\tfrac{3\pi\theta}{2\omega} $\\[5pt]$\mathbb{I}=\left\{\frac{1}{2},\frac{3}{2}\right\}$, if $\frac{3\pi}{2}<\omega\leq2\pi$} & $\mathbb{I}=\{1\}$\\
			\bottomrule
		\end{tabular}
	\end{threeparttable}
\end{table}

We employ a smooth cut-off function $\eta_\rho$, defined for $\rho\in(0,2]$ by  \cite{cai2001finite,cai2001solution}
\begin{align}\label{cutoff}
	\eta_{\rho}(r_j)=\left\{\begin{aligned}
		1,&& 0<r_j<\frac{\rho R}{2},\\
		\tfrac{15}{16}(\tfrac{8}{15}-\left(\tfrac{4r_j}{\rho R}-3\right)+\tfrac{2}{3}(\tfrac{4r_j}{\rho R}-3)^3-\tfrac{1}{5}(\tfrac{4r_j}{\rho R}-3)^5),&& \tfrac{\rho R}{2}\leq r_j<\rho R,\\
		0,&& r_j\geq\rho R,\\
	\end{aligned}\right.
\end{align}
where $R\in\mathbb{R}_+$ is a fixed number so that $ \eta_{\rho}$
vanishes identically on $ \partial\Omega $. In practice, we take $R$ to be small enough so that when $i\neq j$, the support of $\eta_\rho(r_i)$ does no intersect with that of $\eta_\rho(r_j)$. By construction, we have $\eta_{\rho}\in C^2([0,\infty))$. This choice of $\eta_\rho(r_j)$ preserves the zero boundary condition and also enjoys good regularity (and thus does not introduce extra singularities). In practice, any choice satisfying these desirable properties is valid. Then the solution $u$ of problem \eqref{problem} has the following singular function representation \cite{I1984The,DaugeM,Le2008e}
\begin{equation}\label{split}
u=w+\sum_{j=1}^{M}\sum_{i\in\mathbb{I}_j}\gamma_{j,i}\eta_{\rho_j}(r_j) s_{j,i}(r_j,\theta_j),\quad\mbox{with } w\in H^2(\Omega)\cap H_0^1(\Omega),
\end{equation}
where the scalars $\gamma_{j,i}\in\mathbb{R}$ are known as stress intensity factors and given by the following extraction formulas \cite[Lemma 8.4.3.1]{grisvard2011elliptic}:
\begin{equation}\label{eqn:extraction}
	\gamma_{j,i}=\frac{1}{ i\pi}\left(\int_\Omega f\eta_{\rho_j} s_{j,-i}{\rm d}x+\int_\Omega u\Delta(\eta_{\rho_j}s_{j,-i}){\rm d} x\right),
\end{equation}
where $s_{j,-i}$ denotes dual singular functions of $s_{j,i}$. Specifically, if $ s_{j,i}(r_j,\theta_j) =r_j^{\frac{i\pi}{\omega_j}}\sin \frac{i\pi\theta_j}{\omega_j} $, the dual function $s_{j,-i}$ is given by
$s_{j,-i}(r_j,\theta_j)=r_j^{-\frac{i\pi}{\omega_j}}\sin \frac{i\pi\theta_j}{\omega_j}$ \cite[p. 2639]{cai2006finite}. Moreover, the following regularity estimate on the regular part $w$ holds  \cite{cai2006finite,grisvard2011elliptic}
\begin{equation}\label{eqn:2D-apriori}
\|w\|_{H^2(\Omega)} + \sum_{j=1}^M \sum_{i\in \mathbb{I}_j}|\gamma_{j,i}| \leq c\|f\|_{L^2(\Omega)}.
\end{equation}

\subsubsection{Singularity enriched physics-informed neural network}
Now we propose  singularity enriched PINN (SEPINN) for problem \eqref{problem}, inspired by the representation \eqref{split}. The idea of singularity enrichment / splitting has been widely used in the context of FEM (see the introduction for details), and also in SSDRM \cite{HuJinZhou:2022}, which splits out explicitly the singular part due to point / line sources. We discuss only the case with one singular function in $\Omega$ satisfying the condition \eqref{condition}. The case of multiple singularities can be handled similarly. With $S=\gamma\eta_\rho s$, the regular part $w$ satisfies
\begin{align}\label{modif pro}
	\left\{\begin{aligned}
		-\Delta w &= f + \gamma\Delta(\eta_\rho s),&& \mbox{in }\Omega,\\
		w&=0,&& \mbox{on }\Gamma_D,\\
		\partial_n w&=0,&& \mbox{on }\Gamma_N,
	\end{aligned}\right.
\end{align}
where $\gamma\in\mathbb{R}$ is unknown. Since $w\in H^2(\Omega)$, it can be well approximated using PINN. The parameter $\gamma$ can be either learned together with $w$ or extracted from $w$ via \eqref{eqn:extraction}.

Based on the principle of PDE residual minimization, the solution $w^*$ of \eqref{modif pro} and the exact parameter $\gamma^*$ in \eqref{split} is a global minimizer of the following loss
\begin{equation}\label{prob}
	\mathcal{L}_{\boldsymbol\sigma}(w,\gamma)=\|\Delta w +f+ \gamma\Delta(\eta_\rho s)\|_{L^2(\Omega)}^2+\sigma_d\|w\|_{L^2(\Gamma_D)}^2+\sigma_n\left\|\partial_n w\right\|_{L^2(\Gamma_N)}^2,
\end{equation}
where the penalty weights $\boldsymbol\sigma=(\sigma_d,\sigma_n)\in\mathbb{R}^2_+$ are tunable. Following Section \ref{ssec:pinn}, we employ a DNN $w_\theta\in \mathcal{A}$ to approximate  $w^*\in H^2(\Omega)$, and treat the parameter $\gamma$ as a trainable parameter and learn it along with the DNN parameters $\theta$. This leads to an empirical loss
\begin{equation}\label{2deqn:loss-emp}
	\begin{aligned}
\widehat{\mathcal{L}}_{\boldsymbol\sigma}(w_\theta,\gamma)=&\dfrac{|\Omega|}{N_r}\sum_{i=1}^{N_r}(\Delta w_\theta(X_i)+f(X_i)+\gamma\Delta(\eta_\rho s)(X_i))^2\\
&+\sigma_d\dfrac{|\Gamma_D|}{N_d}\sum_{j=1}^{N_d}w_\theta^2(Y_j)+\sigma_n \dfrac{|\Gamma_N|}{N_n}\sum_{k=1}^{N_n}\left(\partial_n w_\theta(Z_k)\right)^2,
	\end{aligned}
\end{equation}
with i.i.d. sampling points $\{X_i\}_{i=1}^{N_r}\sim U(\Omega)$, $\{Y_j\}_{j=1}^{N_d}\sim U(\Gamma_D)$ and
$\{Z_k\}_{k=1}^{N_n}\sim U(\Gamma_N)$. Let $(\widehat \theta^*,\widehat \gamma^*)$ be a minimizer of the
empirical loss $ \widehat{L}(w_\theta,\gamma)$. Then $w_{\widehat{\theta^*}}\in\mathcal{A}$ is the DNN
approximation of the regular part $w^*$, and the approximation $\hat u$ to $u$ is given by
$\hat u = w_{\hat \theta^*} + \hat\gamma^* \eta_\rho s$.

Now we discuss the training of the loss $\widehat{\mathcal{L}}_{\boldsymbol{\sigma}}(w_\theta,\gamma) $. One can minimize $\widehat{\mathcal{L}}_{\boldsymbol{\sigma}}(w_\theta,\gamma)$ directly with respect to $\theta$ and $\gamma$, which works reasonably. However, the DNN approximation $\widehat u$ tends to have larger errors on the boundary $\partial\Omega$ than in the domain $\Omega$, but the estimated $ \widehat\gamma^* $ is often accurate. Thus we adopt a two-stage training procedure.
\begin{itemize}
  \item[(i)] At Stage 1, minimize the loss $\widehat{\mathcal{L}}_{\boldsymbol\sigma}(w_\theta,\gamma)$ for a fixed $\boldsymbol\sigma$, and obtain the minimizer $(\widehat\theta^*,\widehat \gamma^*)$.
  \item[(ii)] At Stage 2, fix $\gamma$ in $\widehat{\mathcal{L}}_{\boldsymbol\sigma}(w_\theta,\gamma)$ at $\widehat \gamma^*$, and learn $\theta$ via a path-following strategy \cite{LuJohnson:2021,HuJinZhou:2022}.
\end{itemize}
The estimate $\widehat{\gamma}^*$ depends on the choice of $\boldsymbol\sigma$, but numerically it does not vary much with the choice. Moreover, the optimal $\widehat\gamma^*$ may vanish, i.e., the solution $u$ is smooth.

Now we describe a path-following strategy to update $\boldsymbol{\sigma}$. We start with small values $\boldsymbol\sigma^{(1)}=\boldsymbol\sigma$. After each loop (i.e., finding one minimizer $\widehat\theta_k^*$), we update $\boldsymbol{\sigma}$ geometrically: $\boldsymbol{\sigma}^{(k+1)}=q\boldsymbol{\sigma}^{(k)}$, with
$q>1$. By updating $\boldsymbol{\sigma}^{(k)}$, the minimizer $\widehat{\theta}^*_k$ of the loss $\widehat{\mathcal{L}}_{\boldsymbol{\sigma}^{(k)}}(w_\theta,\widehat{\gamma}^*)$ also approaches that of problem \eqref{prob}, and the path-following strategy enforces the boundary conditions progressively, which is beneficial to obtain good approximations, since when $\boldsymbol{\sigma}$ is large, the optimization problem is known to be numerically stiff. Note that the minimizer $\widehat\theta_{k+1}^*$ of the $\boldsymbol{\sigma}^{(k+1)}$-problem (i.e., minimizing $\widehat{\mathcal{L}}_{\boldsymbol\sigma^{(k+1)}}(w_\theta,\widehat{\gamma}^*)$) can be initialized to $\widehat{\theta}_k^*$ of the $\boldsymbol\sigma^{(k)}$-problem to warm start the training process. Hence, for each fixed $\boldsymbol\sigma^{(k)}$ (except $\boldsymbol\sigma^{(1)}$), the initial parameter configuration is close to the optimal one, and the training only requires few iterations to reach convergence. The overall procedure is shown in Algorithm \ref{algorithm:2d} for 2D problems with corner singularities and / or mixed boundary conditions. The stopping condition at line 3 of the algorithm reads: the iteration terminates whenever the maximum parameter vector $\boldsymbol{\sigma}^*$  is reached or the validating error falls below a certain threshold.

\begin{algorithm}[H]
\caption{SEPINN for 2D problems}\label{algorithm:2d}
\begin{algorithmic}[1]
\State Set $\boldsymbol\sigma^{(1)}$, and obtain the minimizer $(\widehat{\theta}^*,\widehat{\gamma}^*)$ of the loss $\widehat{\mathcal{L}}_{\boldsymbol\sigma^{(1)}}(w_\theta,\gamma)$.
\State Set $k=1$, $\widehat{\theta}_{0}^*=\widehat{\theta}^*$, and increasing factor $q>1$.
\While{Stopping condition not met}
	\State Find a minimizer $\widehat{\theta}_{k}^*$ of the loss $\widehat{\mathcal{L}}_{\boldsymbol\sigma^{(k)}}(w_\theta,\widehat{\gamma}^*)$ (initialized to $\widehat{\theta}_{k-1}^*$).
   \State Update $\boldsymbol\sigma$ by $\boldsymbol\sigma^{(k+1)} =q\boldsymbol\sigma^{(k)}$, and $ k \leftarrow k+1$.
\EndWhile
\State Output the SEPINN approximation $\widehat{u}=w_{\widehat{\theta}_{k-1}^*}+\widehat{\gamma}^*\eta s$.
	\end{algorithmic}
\end{algorithm}

PINN exhibits the so-called frequency principle during the training: the low-frequency components are learned at a faster rate than the high-frequency ones \cite{WangPerdikaris:2022jcp}. In this vein, SEPINN harnesses this fact: by extracting the singular part, the regular part is smoother and thus by the frequency bias, it easier to learn.
\begin{remark}
In this work we have focused on enforcing boundary conditions with the penalty method, due to its ease of implementation. {Yet the penalty approach} incurs the consistency error and restricts the norm in which the error can be bounded. In practice one may also enforce the boundary conditions exactly using {\rm(}approximate{\rm)} distance functions for domains with simple geometry \cite{SukumarSrivastava:2022}.
However,  {the construction of distance functions for a general polygonal domain is fairly complicated}, and thus we do not pursue the idea in this work.
\end{remark}

\subsection{Three-dimensional problem}\label{3d domain}
Now we develop SEPINN for the 3D Poisson problem with edge singularities.
\subsubsection{Singular function representation}\label{3d:singular repre}
There are several different types of geometric singularities in the 3D case, and each case has to be dealt with separately \cite{1999Singularities}. We only study edge singularities, which cause strong solution singularities. Indeed, the $H^2(\Omega)$-regularity of the solution $u$ of problem \eqref{problem} is not affected by the presence of conic points  \cite{1990Singularities,grisvard2011elliptic,KozlovMazya:2001}. Now we state the precise setting. Let $\Omega_0\subset \mathbb{R}^2$ be a polygonal domain as in Section \ref{2d domain}, and $\Omega=\Omega_0\times(0,l)$.
Like before, let $\boldsymbol{v}_j,j=1,2,\cdots,M$, be the vertices of $\Omega_0$ whose interior angles $\omega_j,j=1,2,\cdots,M,$ satisfy \eqref{condition}, and let $\Gamma_{z_1}=\Omega_0\times\{0\}$ and $\Gamma_{z_2}=\Omega_0\times\{l\}$. We employ a complete orthogonal system $\{Z_{j,n}\}_{n=0}^\infty$ of $ L^2(0, l) $, given in Table \ref{table 2.3}.

\begin{table}[h]
	\centering
	\setlength{\tabcolsep}{5pt}\begin{threeparttable}  \caption{The orthogonal basis $\{Z_{j,n}(z)\}_{n=0}^\infty$ of $L^2(0,l)$, for the $j$th wedge, with different boundary conditions on the surfaces $\Gamma_{z_1}$ and $\Gamma_{z_2}$.\label{table 2.3}}
		\begin{tabular}{c|c|c}
			\toprule
			{$\Gamma_{z_2}\backslash \Gamma_{z_1}$} & Dirichlet & Neumann \\
			\midrule
			Dirichlet & \makecell{$Z_{j,n}(z)=\sin(\xi_{j,n} z)$,\\ $\xi_{j,n}=\frac{n\pi}{l}$, $n\in \mathbb{N}\cup\{0\}$} & 		\makecell{$ Z_{j,0}(z)=0$,\\[5pt]$ Z_{j,n}(z)=\cos(\xi_{j,n} z)$, \\[5pt] $\xi_{j,n}=(n-\tfrac{1}{2}) \tfrac{\pi}{l}$, $n\in \mathbb{N}$}    \\
			\midrule
			Neumann &
			\makecell{$ Z_{j,0}(z)=0$,\\[5pt]$ Z_{j,n}(z)=\sin(\xi_{j,n}z)$, \\[5pt] $\xi_{j,n}=(n-\tfrac{1}{2}) \tfrac{\pi}{l}$, $n\in\mathbb{N}$}
			&
			\makecell{$Z_{j,n}(z)=\cos(\xi_{j,n}z)$, \\ $\xi_{j,n}=\tfrac{n\pi}{l}$, $n\in \mathbb{N}\cup\{0\}$}\\
			\bottomrule
		\end{tabular}
	\end{threeparttable}
\end{table}

Now we assume that near each edge $\boldsymbol{v}_j \times(0,l) $, the domain $ \Omega $ coincides with a 3D wedge
$G_j=\{(x_j,y_j,z)\in\Omega:0<r_j<R_j,0<\theta_j<\omega_j,0<z<l\},$
where $ (r_j, \theta_j) $ are local polar coordinates (of the local Cartesian coordinates $ (x_j, y_j) $), cf. Fig. \ref{fig:sketch}(b) for a sketch. The functions $ u\in L^2(\Omega) $ and $ f\in L^2(\Omega) $ from problem \eqref{problem} can then be represented by the following convergent Fourier series in the 3D wedge $G_j$ (by suppressing the subscript $j$):
\begin{align}\label{fourier u 3d}
	u(x,y,z)&=\frac{1}{2}u_0(x,y)Z_0(z)+\sum_{n=1}^\infty u_n(x,y)Z_n(z),\\
\label{fourier f 3d}
	f(x,y,z)&=\frac{1}{2}f_0(x,y)Z_0(z)+\sum_{n=1}^\infty f_n(x,y)Z_n(z),
\end{align}
where $\{u_n\}_{n\in\mathbb{N}}$ and $\{f_n\}_{n\in\mathbb{N}}$ are defined on the 2D domain $\Omega_0$ by
\begin{align*}
    u_n(x,y)=\dfrac{2}{l}\int_0^lu(x,y,z)Z_n(z){\rm d}z
    \quad \mbox{and}\quad
 f_n(x,y)=\dfrac{2}{l}\int_0^lf(x,y,z)Z_n(z){\rm d}z .
\end{align*}
Substituting \eqref{fourier u 3d} and \eqref{fourier f 3d} into \eqref{problem} gives countably many 2D elliptic problems:
\begin{align}\label{3d:fourier equation}
	\left\{\begin{aligned}
		-\Delta u_n+\xi_{j,n}^2u_n&=f_n,&& \mbox{in }\Omega_0,\\
		u_n&=0,&& \mbox{on }\Gamma_D,\\
		\partial_n u_n&=0,&& \mbox{on }\Gamma_N.
	\end{aligned}\right.
\end{align}

Then problem \eqref{problem} can be analyzed via the 2D problems. Below we describe the edge behavior of the weak solution $ u\in H^1(\Omega) $ \cite[Theorem 2.1]{nkemzi2021singular}. The next theorem gives a crucial decomposition of $u\in H^{1+\frac{\pi}{\omega^*}-\epsilon}$ for every $\epsilon>0$. The functions $ \Psi_{j,i} $ are the so-called edge flux intensity functions.
\begin{theorem}\label{theorem:3d split}
For any fixed $f\in L^2(\Omega)$, let $u\in H^1(\Omega)$ be the unique weak solution to  problem \eqref{problem}. Then there exist unique functions $\Psi_{j,i}\in H^{1-\lambda_{j,i}}(0,l)$ of the variable $z$ such that $ u $ can be split into a sum of a regular part $ w \in H^2(\Omega)$ and a singular part $ S $ with the following properties:
\begin{subequations}
    \begin{align}\label{3dsplit}
		u&=w+S,\quad S=\sum_{j=1}^M\sum_{i\in\mathbb{I}_j}S_{j,i}(x_j,y_j,z),\\
        \label{3d singular fun}
        	S_{j,i}(x_j,y_j,z)&=(T_j(r_j,z)*\Psi_{j,i}(z))\eta_{\rho_j}(r_j)s_{j,i}(r_j,\theta_j),
        \end{align}
	\end{subequations}
where the functions $ T_j $ are fixed Poisson's kernels,  and the symbol $ * $ denotes the convolution in $ z\in(0,l) $. Moreover, there exists a constant $ C >0  $ independent of
$ f\in L^2(\Omega) $ such that
$\|w\|_{H^2(\Omega)} \leq C\|f\|_{L^2(\Omega)}$.
\end{theorem}

The explicit form of the function $ T_j(r_j,z)*\Psi_{j,i}(z) $ and the formula for the coefficients $\gamma_{j,i,n}$ are given below \cite[Theorem 2.2]{nkemzi2021singular} \cite[pp. 179--182]{2012Flux}.
\begin{theorem}\label{theorem: 3dsplit2}
The coefficients $ \Phi_{j,i}(x_j,y_j,z)=T_j(r_j,z)*\Psi_{j,i}(z) $ of the singularities in \eqref{3d singular fun} can be represented by Fourier series in $ z $ and with respect to the orthogonal system $ \{Z_{n}(z) \}_{n=0}^\infty $:
\begin{align}
  T_j(r_j,z)&=\dfrac{1}{2}Z_0(z)+\sum_{n=1}^\infty{\rm e}^{-\xi_{j,n}r_j}Z_n(z),\nonumber\\
  \Psi_{j,i}(z)&=\dfrac{1}{2}\gamma_{j,i,0}Z_0(z)+\sum_{n=1}^\infty\gamma_{j,i,n}Z_n(z),\nonumber\\
  \Phi_{j,i}(x_j,y_j,z)&=\dfrac{1}{2}\gamma_{j,i,0}Z_0(z)+\sum_{n=1}^\infty\gamma_{j,i,n}{\rm e}^{-\xi_{j,n}r_j}Z_n(z),\label{expansion:phi}
\end{align}
where $ \gamma_{j,i,n} $ are given by
$ \gamma_{j,i,n}=\frac{2}{l\omega_j\lambda_{j,i}}\int_{G_j}f_j^*{\rm e}^{\xi_{j,n}r_j}s_{j,-k}(r_j,\theta_j)Z_n(z){\rm d}x{\rm d}y{\rm d}z$, and
$f_j^*=f\eta_{\rho_j}-u(\frac{\partial^2\eta_{\rho_j}}{\partial r_j^2}+(2\xi_{j,n}+\frac{1}{r_j})\frac{\partial\eta_{\rho_j}}{\partial r_j}+(2\xi_{j,n}^2+\frac{\xi_{j,n}}{r_j})\eta_{\rho_j})-2\frac{\partial u}{\partial r_j}(\frac{\partial \eta_{\rho_j}}{\partial r_j}+\xi_{j,n}\eta_{\rho_j}).$
Moreover there exists a constant $ C>0 $ independent of $ f $ such that
\begin{equation}
|\gamma_{j,i,0}|^2+\sum_{n=1}^\infty\xi_{j,n}^{2\left(1-\lambda_{j,i}\right)}|\gamma_{j,i,n}|^2\leq C\|f\|_{L^2(G_j)}.
\end{equation}
\end{theorem}
Note that the formula for $\gamma_{j,i,n}$ involves a singular integral and numerically inconvenient to evaluate. We discuss the case of only one edge with one singular function below. Then the solution $u$ can be split into
\begin{equation}\label{split:3d}
	u=w+\Phi\eta_\rho s .
\end{equation}
Since the functions $\Phi\eta_\rho s$ and $\partial_n(\Phi\eta_\rho s)$ vanish on $\Gamma_D$ and $\Gamma_N$, respectively, $w$ solves
\begin{align}\label{3d modified}
	\left\{\begin{aligned}
		-\Delta w &= f + \Delta (\Phi\eta_\rho s),&& \mbox{in }\Omega,\\
		w&=0,&& \mbox{on }\Gamma_D,\\
		\partial_n w&=0,&& \mbox{on }\Gamma_N.
	\end{aligned}\right.
\end{align}
This forms the basis of SEPINN for 3D problems with edge singularities. Below we describe two strategies for constructing SEPINN approximations, i.e., SEPINN-C based on a cutoff approximation of the infinite series and SEPINN-N based on multiple DNN approximations.

\subsubsection{SEPINN -- Cutoff approximation}
The expansion \eqref{expansion:phi} of $\Phi$ in \eqref{3d modified} involves infinitely many unknown scalar coefficients $\{\gamma_{j,i,n}\}_{n=0}^\infty$. In practice, it is infeasible to learn all of them. However, since the series is convergent, we may truncate it to a finite number of terms: the edge flux intensity function $ \Phi_{j,i}(x_j,y_j,z) $ from  \eqref{expansion:phi} is approximated by the truncation
\begin{equation}\label{expansion}
	\Phi_{j,i}^N(x_j,y_j,z)=\dfrac{1}{2}\gamma_{j,i,0} Z_0(z)+\sum_{n=1}^N\gamma_{j,i,n}{\rm e}^{-\xi_{j,n}r_j}Z_n(z), \quad \mbox{with } N\in \mathbb{N}.
\end{equation}
The approximate singular function $S_{j,i}^N$ is given by
\begin{equation}\label{truncating}
	S_{j,i}^N(x_j,y_j,z)=\Phi_{j,i}^N(x_j,y_j,z)\eta_{\rho_j}(r_j)s_{j,i}(r_j,\theta_j).
\end{equation}
In view of the splitting \eqref{3dsplit} and truncation \eqref{truncating}, the approximate regular part $w$ satisfies
with $\tilde f = f + \gamma_0\Delta(\eta_\rho s) + \sum_{n=1}^N \gamma_n\Delta({\rm e}^{-\xi_n r}Z_n(z)\eta_\rho s)$,
\begin{align}\label{3d modif pro}
	\left\{\begin{aligned}
		-\Delta w &= \tilde f,&& \mbox{in }\Omega,\\
		w&=0,&& \mbox{on }\Gamma_D,\\
		\partial_nw&=0,&& \mbox{on }\Gamma_N,
	\end{aligned}\right.
\end{align}
with $\gamma_i, i=0,1,\cdots,N$, being $N+1$ unknown parameters. Let $w^N$ be the solution of \eqref{3d modif pro} and $w^*$ the solution of \eqref{3d modified}. The next result shows that when $N$ is large enough, the error $w^N-w^*$ can be made small, and so is the error between $u^N$ and $u^*$, which underpins the truncation method.
\begin{theorem}\label{theorem:trun}
Let	$w^N$ be the solution of \eqref{3d modif pro} and $S^N$ be the truncated singular function. Let $u^N=w^N+S^N$ and let $u^*$ be the solution of \eqref{problem}. Then there holds
\begin{equation}
		\|u^N-u^*\|_H^1( \Omega)\leq CN^{-1}\|f\|_{L^2(G)}.
\end{equation}
\end{theorem}
\begin{proof}
By \cite[Lemma 3.3]{nkemzi2021singular}, there exists a constant $ C >0 $ independent of $f$ such that
\begin{equation}\label{inequality:s}
\|S^N-S^*\|_{H^1(G)}\leq CN^{-1}\|f\|_{L^2(G)}.
\end{equation}
Since $w^N$ solves \eqref{3d modif pro} and $w^*$ solves \eqref{3d modified}, $w^*-w^N$ satisfies problem \eqref{problem}
with the source $f^N=\sum_{n=N+1}^\infty \gamma_n\Delta({\rm e}^{-\xi_n r}Z_n(z)\eta_\rho s)$. By elliptic regularity theory, we have
\begin{align}		
\|w^*-w^N\|_{H^1(\Omega)}\leq C\|f^N\|_{H_{\Gamma_D}^{-1}(\Omega)}
		\leq C\|\Delta (S^N-S^*)\|_{H_{\Gamma_D}^{-1}(\Omega)},
\end{align}
where the notation $H_{\Gamma_D}^{-1}(\Omega)$ denotes the dual space of $H_{\Gamma_D}^1(\Omega)=\{v\in H^1(\Omega): v=0 \mbox{ on }\Gamma_D\}$. Upon integration by parts, we get
\begin{align*}
		&\|\Delta (S^N-S^*)\|_{H_{\Gamma_D}^{-1}(\Omega)}=\sup_{v\in H_{\Gamma_D}^1(\Omega), \|v\|_{H^1(\Omega)}\leq 1}\left|\int_\Omega\Delta (S^N-S^*)v{\rm d}\mathbf{x}\right|\\
		=&\sup_{v\in H_{\Gamma_D}^1(\Omega),\|v\|_{H^1(\Omega)}\leq 1}\left|\int_\Omega\nabla (S^N-S^*)\cdot\nabla v{\rm d} {x}\right|
		\leq\|S^N-S^*\|_{H^1(\Omega)}
        =\|S^N-S^*\|_{H^1(G)}.
\end{align*}
Combining this estimate with \eqref{inequality:s} yields the desired assertion.
\end{proof}

Following Section \ref{ssec:pinn}, we approximate  $w^N$ by an element $w_\theta \in \mathcal{A}$. Like in the 2D case, we view the parameters $\boldsymbol\gamma_N:=( \gamma_0, \ldots,\gamma_N)$ as trainable parameters and learn them along with DNN parameters $\theta$.
Thus, we employ the following empirical loss
\begin{equation*}
	\begin{aligned}
		\widehat{\mathcal{L}}_{\boldsymbol\sigma}(w_\theta;\boldsymbol{\gamma}_N)\!\!=\frac{|\Omega|}{N_r}\!\sum_{i=1}^{N_r}\!(\Delta w_\theta(X_i)\!\!+\!\!\tilde f(X_i))^2\!+\!\sigma_d\frac{|\Gamma_D|}{N_d}\!\sum_{j=1}^{N_d}\!w_\theta^2(Y_j)\!\!+\!\!\sigma_n\frac{|\Gamma_N|}{N_n}\!\sum_{k=1}^{N_n}\!\!\left({\partial_n w_\theta}(Z_k)\right)^2,
	\end{aligned}
\end{equation*}
with i.i.d. sampling points $\{X_i\}_{i=1}^{N_r}\sim U(\Omega)$, $\{Y_j\}_{j=1}^{N_d}\sim U(\Gamma_D)$ and $\{Z_k\}_{k=1}^{N_n}\sim U(\Gamma_N)$. The empirical loss $\widehat{\mathcal{L}}_{\boldsymbol\sigma}(w_\theta;\boldsymbol{\gamma}_N) $ can be minimized using the
two-stage procedure as in the 2D case.

\subsubsection{SEPINN -- Neural networks approximation}
There is actually a direct way to resolve the singular part $S$, i.e., using a DNN to approximate $\Phi$ in \eqref{3dsplit}. We term the resulting method SEPINN-N. This strategy eliminates the necessity of explicitly knowing the expansion basis and relieves us from lengthy derivations. Thus, compared with SEPINN-C, it is more direct and simpler to implement. The downside is an increase in the number of parameters that need to be learned. Specifically, let $\mathcal{B}$ to be a DNN function class with a fixed architecture (possibly different from $\mathcal{A}$) and $\zeta$ its parameterization.  Then a DNN $\Phi_\zeta\in\mathcal{B}$ is employed to approximate $\Phi$ in \eqref{split:3d}, where the DNN parameters $\zeta$ are also learned. The splitting is then given by
\begin{equation}
	u=w+\Phi_\zeta\eta_\rho s.
\end{equation}
Since we cannot guarantee $\Phi_\zeta=0$ on $\Gamma_D$ or $\partial_n\Phi_\zeta=0$ on $\Gamma_N$, the boundary conditions of $w$ have to be modified accordingly (noting $ {\partial_n (\eta_\rho s)}=0 $):
\begin{align}
	\left\{\begin{aligned}
		-\Delta w &= f + \Delta (\Phi_\zeta\eta_\rho s),&& \mbox{in }\Omega,\\
		w&=-\Phi_\zeta\eta_\rho s,&& \mbox{on }\Gamma_D,\\
		\partial_n w&=-\partial_n(\Phi_\zeta)\eta_\rho s,&& \mbox{on }\Gamma_N.
	\end{aligned}\right.
\end{align}
Like before, we can obtain the following empirical loss
\begin{equation}\label{eqn2:loss-emp}
	\begin{aligned}
		&\widehat{\mathcal{L}}_{\boldsymbol\sigma}(w_\theta;\Phi_\zeta)=\dfrac{|\Omega|}{N_r}\sum_{i=1}^{N_r}\left(\Delta w_\theta(X_i)+ f(X_i) +\Delta(\Phi_\zeta\eta_\rho s)(X_i)\right)^2\\
  &+\sigma_d\dfrac{|\Gamma_D|}{N_d}\sum_{j=1}^{N_d}(w_\theta(Y_j)+\Phi_\zeta\eta_\rho s(Y_j))^2
		+\sigma_n \dfrac{|\Gamma_N|}{N_n}\sum_{k=1}^{N_n}\left(\partial_n w_\theta(Z_k)+\partial_n(\Phi_\zeta)\eta_\rho s(Z_k)\right)^2,
	\end{aligned}
\end{equation}
with i.i.d. sampling points $\{X_i\}_{i=1}^{N_r}\sim U(\Omega)$, $\{Y_j\}_{j=1}^{N_d}\sim U(\Gamma_D)$ and $\{Z_k\}_{k=1}^{N_n}\sim U(\Gamma_N)$.
The implementation of SEPINN-N is direct, since both DNNs $w_\theta$ and $\Psi_\zeta$ are learned, and the resulting optimization problem can be minimized directly.

\section{Error analysis}\label{sec:error}
Now we discuss the error analysis of SEPINN developed in Section \ref{sec:SEPINN}, following the strategies established in the recent works \cite{JiaoLai:2022cicp,LuChenLu:2021,HuJinZhou:2022}, in order to provide theoretical guarantee of SEPINN. We only analyze the 2D problem in Section \ref{2d domain}. Let $u^* = w^* + \gamma^* \eta_\rho s$ be the exact solution to problem \eqref{problem}, in which case $(w^*,\gamma^*)$ is also a global minimizer of the loss $\mathcal{L}_{\boldsymbol\sigma}(w,\gamma)$ (with a zero loss value).
Moreover, we assume that $w^*\in H^3(\Omega)$ and $|\gamma^*|\le B_\gamma$, cf. \eqref{eqn:2D-apriori}. The $H^3(\Omega)$ regularity of the smooth part $w^*$ is needed for controlling the error of approximating $w$ in the $H^2(\Omega)$ norm using DNNs. If the assumption does not hold, we can split out additional singular function(s); see the argument in Section \ref{2d:singular repre}. Several preliminary results used in the analysis are given in the appendix.

The following approximation property holds \cite[Proposition 4.8]{GuhringRaslan:2021}. The notation $(s=2)$ equals to 1 if $s=2$ and zero otherwise.
\begin{lemma}\label{lem:tanh-approx}
Let $s\in\mathbb{N}\cup\{0\}$ and $p\in[1,\infty]$ be fixed, and $v\in W^{k,p}(\Omega)$ with $k\geq  s+1$. Then for any tolerance $\epsilon>0$, there exists at least one $v_\theta$ of depth $\mathcal{O}\big(\log(d+k)\big)$, with $|\theta|_{\ell^0}$ bounded by $\mathcal{O}\big(\epsilon^{-\frac{d}{k-s-\mu (s=2)}}\big)$ and  $|\theta|_{\ell^\infty}$ by $\mathcal{O}(\epsilon^{-2-\frac{2(d/p+d+s+\mu(s=2))+d/p+d}{k-s-\mu (s=2)}})$, where $\mu>0$ is arbitrarily small, such that
\begin{equation*}
     \|v-v_\theta\|_{W^{s,p}(\Omega)} \leq \epsilon.
\end{equation*}
\end{lemma}

For $\epsilon>0$, with $d=2$, $k=3$ and $s=2$, by Lemma \ref{lem:tanh-approx},  there exists a DNN
\begin{equation}\label{eqn:NN-set}
v_\theta \in \mathcal{N}_ \varrho(C, C\epsilon^{-\frac{2}{1-\mu}},C\epsilon^{-2-\frac{16+2\mu}{1-\mu}}) =: \mathcal{W}_\epsilon,
\end{equation}
such that $ \|w^*-v_\theta  \|_{H^2(\Omega)} \le \epsilon$.
Also let $I_\gamma = [-B_\gamma,B_\gamma]$. Let $(\widehat w_\theta,\widehat \gamma)$ be a minimizer of $\widehat{\mathcal{L}}_{\boldsymbol\sigma}(w_\theta,\gamma)$, cf. \eqref{2deqn:loss-emp} over $\mathcal{W}_\epsilon\times I_\gamma$, and set
$\widehat u = \widehat w_\theta + \widehat \gamma \eta_\rho s.$
The next lemma gives a decomposition of the error $\| u^* - \widehat u \|_{L^2(\Omega)}$.

\begin{lemma}\label{lem:err-decomp}
For any $\epsilon>0$,
let $(\widehat w_\theta, \widehat \gamma ) \in \mathcal{W}_\epsilon \times I_\gamma$ be a minimizer to the loss $\widehat{\mathcal{L}}_{\boldsymbol\sigma}(w_\theta,\gamma)$. Then there holds
\begin{equation*}
\begin{aligned}
&\| u^* - \widehat u \|_{L^2(\Omega)}^2  \le
c \min(\sigma_d, \sigma_n)^{-1}\mathcal{L}_{\boldsymbol\sigma}(\widehat w_\theta,\widehat \gamma) \\
 \le & c \min(\sigma_d, \sigma_n)^{-1} \Big((1+\sigma_d+\sigma_n)\epsilon^2
+ \sup_{(w_\theta,\gamma)\in \mathcal{W}_\epsilon\times I_\gamma} | \mathcal{L}_{\boldsymbol\sigma}( w_\theta, \gamma) - \widehat{\mathcal{L}}_{\boldsymbol\sigma}(w_\theta,  \gamma)  |\Big).
\end{aligned}
\end{equation*}
where the constant $c$ is independent of $\boldsymbol\sigma$.
\end{lemma}
\begin{proof}
For any $(w_\theta,\gamma) \in \mathcal{W}_\epsilon \times I_\gamma$,
let $  u = w_\theta + \gamma \eta_\rho s$, and denote the corresponding overall error by $e = u^*-  u$. By the trace theorem, we have
\begin{equation*}
\mathcal{L}_{\boldsymbol\sigma}(w_\theta,\gamma)   = \|  \Delta e \|_{L^2(\Omega)}^2
 + \sigma_d \| e \|_{L^2(\Gamma_D)}^2 + \sigma_n \| \partial_n e \|_{L^2(\Gamma_N)}^2 \le
 c  (1+\sigma_d + \sigma_n)  \| e \|_{H^2(\Omega)}^2.
\end{equation*}
To treat the nonzero boundary conditions of $w_\theta$, we define the extension $\zeta$ by
\begin{equation*}
\left\{\begin{aligned}-\Delta \zeta &=0 , &&\text {in } \Omega,\\
\zeta &=w_\theta , &&\text {on } \Gamma_D,\\
\partial_n \zeta &= \partial_n w_\theta  ,& &\text {on } \Gamma_N.
\end{aligned}\right.
\end{equation*}
Then the following elliptic regularity estimate holds \cite[Theorem 4.2, p. 870]{Berggren:2004}
\begin{equation}\label{eqn:stab-hm}
\|  \zeta  \|_{L^2(\Omega)} \le c \big(\| w_\theta \|_{L^2(\Gamma_D)} + \|\partial_n w_\theta \|_{L^2(\Gamma_N)} \big).
\end{equation}
This estimate bounds the consistency error due to the boundary penalty.
Let $\tilde e = e + \zeta$. Then it satisfies
\begin{equation*}
\left\{\begin{aligned}
-\Delta \tilde e &=  - \Delta e  , &&\text {in } \Omega,\\
\tilde e &= 0 , && \text {on } \Gamma_D,\\
\partial_n \tilde e &= 0, && \text {on } \Gamma_N.
\end{aligned}\right.
\end{equation*}
Since $\Delta e \in L^2(\Omega)$, the Poincar\'e inequality and the standard energy argument imply
$$ \|  \tilde e  \|_{L^2(\Omega)} \le c \| \nabla \tilde e  \|_{L^2(\Omega)} \le c \| \Delta e  \|_{L^2(\Omega)}.$$
This, the stability estimate \eqref{eqn:stab-hm} and the triangle inequality lead to
\begin{equation*}
\|   e  \|_{L^2(\Omega)}^2 \le c\big(\|   \tilde  e  \|_{L^2(\Omega)}^2  +  \|  \zeta  \|_{L^2(\Omega)}^2\big)
\le c \min(\sigma_d, \sigma_n)^{-1}\mathcal{L}_{\boldsymbol\sigma}(w_\theta,\gamma) .
\end{equation*}
This proves the first inequality of the lemma.
Next, by Lemma \ref{lem:tanh-approx} and the assumption $w^* \in H^3(\Omega)$,
there exists  $ w_{\bar \theta} \in \mathcal{W}_\epsilon$ such that
$\| w_{\bar\theta} - w^* \|_{H^2(\Omega)} \le \epsilon$. Let $\bar u = w_{\bar\theta} + \gamma^* \eta_\rho s$. Then we derive
\begin{align*}
\mathcal{L}_{\boldsymbol\sigma}( w_{\bar\theta},\gamma^*) & = \| \Delta(w_{\bar\theta} - w^*) \|_{L^2(\Omega)}^2
   + \sigma_d \| w_{\bar\theta} - w^* \|_{L^2(\Gamma_D)}^2 + \sigma_n \| \partial_n (w_{\bar\theta} - w^*) \|_{L^2(\Gamma_N)}^2\\
&\leq   c (1+\sigma_d + \sigma_n) \| w_{\bar\theta} - w^*  \|_{H^2(\Omega)}^2 \le  c (1+\sigma_d + \sigma_n) \epsilon^2.
\end{align*}
Thus by the minimizing property of $(\widehat w_\theta,\widehat\gamma)$, we arrive at
\begin{equation*}
\begin{aligned}
 \mathcal{L}_{\boldsymbol\sigma}(\widehat w_\theta,\widehat \gamma)
  &\le  \big|\mathcal{L}_{\boldsymbol\sigma}(\widehat w_\theta,\widehat \gamma)  -\widehat{\mathcal{L}}_{\boldsymbol\sigma}(\widehat w_\theta,\widehat \gamma) \big|
      +  \big| \widehat{\mathcal{L}}_{\boldsymbol\sigma}( w_{\bar\theta}, \gamma^*) -  \mathcal{L}_{\boldsymbol\sigma}( w_{\bar\theta};  \gamma^*) \big|
      +   \mathcal{L}_{\boldsymbol\sigma}( w_{\bar\theta},  \gamma^*) \\
  &\le 2 \sup_{(w_\theta,\gamma)\in \mathcal{W}_\epsilon \times I_\gamma} | \mathcal{L}_{\boldsymbol\sigma}( w_\theta, \gamma) - \widehat{\mathcal{L}}_{\boldsymbol\sigma}( w_\theta, \gamma)|
  +   c (1+\sigma_d + \sigma_n)  \epsilon^2.
\end{aligned}
\end{equation*}
This completes the proof of the lemma.
\end{proof}

\begin{remark}
The error estimate in Lemma  \ref{lem:err-decomp} is given in terms of the $L^2(\Omega)$ norm. This is due to the use of the penalty method for treating the zero Dirichlet boundary condition, cf. \eqref{eqn:stab-hm}. This is a form of the consistency error, and it essentially prevents one from obtaining error estimates in the $H^1(\Omega)$ norm. The estimate indicates that one should take the parameter $\boldsymbol\sigma$ sufficiently large in order to have small consistency errors. See also the works \cite{JiaoLai:2021,MullerZeinhofer:2022,HuJinZhou:2022} for detailed discussions on the consistency error due to penalization for the Dirichlet boundary condition in the context of the deep Ritz method.
\end{remark}

Next, we bound the error $\mathcal{E}_{stat}=\sup_{(w_\theta,\gamma)\in \mathcal{W}_\epsilon \times I_\gamma} | \mathcal{L}_{\boldsymbol\sigma}( w_\theta, \gamma) - \widehat{\mathcal{L}}_{\boldsymbol\sigma}(w_\theta, \gamma)|$, which arises from approximating the integrals by Monte Carlo. By the triangle inequality, we have the splitting
\begin{equation}\label{eqn:err-sta}
\begin{split}
\mathcal{E}_{stat}
 &\le \sup_{(w_\theta,\gamma)\in \mathcal{W}_\epsilon \times I_\gamma} |\Omega|\Big| \frac{1}{N_r} \sum_{i=1}^{N_r} h_r(X_i;w_\theta,\gamma) - \mathbb{E}_X (h_r(X;w_\theta,\gamma)) \Big|\\
 &\quad +  \sup_{w_\theta\in \mathcal{W}_\epsilon}  \sigma_d |\Gamma_D| \Big|  \frac{1}{N_d} \sum_{j=1}^{N_d} h_d(Y_j;w_\theta) - \mathbb{E}_Y (h_d(Y;w_\theta))  \Big|\\
  &\quad +  \sup_{w_\theta\in \mathcal{W}_\epsilon}  \sigma_n |\Gamma_N| \Big|  \frac{1}{N_n} \sum_{k=1}^{N_n} h_n(Z_k;w_\theta) - \mathbb{E}_Y (h_n(Z;w_\theta))  \Big|,
 \end{split}
\end{equation}
with
$h_r(x;w_\theta,\gamma) = (\Delta w_\theta +f+ \gamma\Delta(\eta_\rho s))^2(x)$ for $x\in \Omega$, $h_d(y;w_\theta) = |w_\theta(y)|^2$ for $y \in \Gamma_D$ and $h_n(z;w_\theta) = |\partial_n w_\theta(z)|^2 $ for $z\in \Gamma_N$. Thus, we define the following three function classes
$\mathcal{H}_r = \{h_r(w_\theta,\gamma): w_\theta \in  \mathcal{W}_\epsilon, \gamma\in I_\gamma\},$ $ \mathcal{H}_d  = \{h_d(w_\theta): w_\theta \in  \mathcal{W}_\epsilon  \}$ and $\mathcal{H}_n  = \{h_n(w_\theta): w_\theta \in  \mathcal{W}_\epsilon\}$.

To bound the error components in the decomposition \eqref{eqn:err-sta}, we employ Rademacher complexity of the DNN function classes $\mathcal{H}_r$, $\mathcal{H}_d$ and $\mathcal{H}_n$, which is then bounded using Dudley's formula in Lemma \ref{lem:Dudley} and Lipschitz continuity of the functions in these DNN function classes in Lemma \ref{lem:fcn-Lip}. These technical details are given in the appendix. Then we can state the following bound on the quadrature error.
\begin{lemma}\label{lem:err-stat}
For any small $\tau $, with probability at least
$1- 3 \tau$, there holds
\begin{equation*}
\sup_{( w_\theta,\gamma)\in\mathcal{W}_\epsilon \times I_\gamma} | \mathcal{L}_{\boldsymbol\sigma}( w_\theta, \gamma) - \widehat{\mathcal{L}}_{\boldsymbol\sigma}( w_\theta,  \gamma)  | \le c(e_r + \sigma_d e_d + \sigma_n e_n),
\end{equation*}
with $c=c(\|f\|_{L^\infty(\Omega)},\|\Delta (\eta_\rho s)\|_{L^\infty(\Omega)})$, and $e_r$, $e_d$ and $e_n$  defined by
\begin{align*}
  e_r & \le c \frac{L^2 B_\theta^{4L} N_\theta^{4L-4} \big(N_\theta^\frac12\big( \log^\frac12 B_\theta + \log^\frac12 N_\theta + \log^\frac12 N_r)  + \log^\frac12 \frac{1}{\tau}\big)}{\sqrt{N_r}}, \\
 e_d & \le c \frac{B_\theta^2 N_\theta^2 \big(N_\theta^{\frac12} (\log^\frac12 B_\theta+ \log^\frac12 N_\theta + \log^\frac12 N_d) + \log^\frac12\frac{1}{\tau} \big)}{\sqrt{N_d}}, \\
 e_n &\le c \frac{B_\theta^{2L} N_\theta^{2L-2} \big(N_\theta^{\frac12} (\log ^\frac12B_\theta + \log^\frac12 N_\theta + \log^\frac12 N_n) + \log^\frac12 \frac{1}{\tau} \big)}{\sqrt{N_n}}.
 \end{align*}
\end{lemma}
\begin{proof}
Fix $m \in \mathbb{N }$, $B_\theta \in [1, \infty)$, $\epsilon \in  (0,1)$,
and $\mathbb{B}_{B_\theta} := \{x\in\mathbb{R}^m:\ |x|_{\ell^\infty}\leq B_\theta\}$. Then by \cite[Prop. 5]{CuckerSmale:2002},
$ \log \mathcal{C}(\mathbb{B}_{B_\theta},|\cdot|_{\ell^\infty},\epsilon)\leq m\log (4B_\theta\epsilon^{-1})$.
The Lipschitz continuity estimates in Lemmas \ref{lem:NN-Lip} and \ref{lem:fcn-Lip} imply
\begin{align*}
 \log \mathcal{C}(\mathcal{H}_r,\|\cdot\|_{L^{\infty}(\Omega)},\epsilon)&\leq
\log \mathcal{C}(\Theta \times I_\gamma,\max(|\cdot|_{\ell^\infty},|\cdot|) ,\Lambda_{r}^{-1}\epsilon)  \leq cN_\theta \log(4B_\theta \Lambda_r\epsilon^{-1}),
\end{align*}
with $\Lambda_{r} = c N_\theta L^3 W^{5L-5}B_\theta^{5L-3}$. By Lemma \ref{lem:fcn-Lip},
we have $M_{\mathcal{H}_r} = c L^2 W^{4L-4}B_\theta^{4L}$. Then setting $s=n^{-\frac12}$ in Lemma \ref{lem:Dudley} and using the facts $1\leq B_\theta$, $1\leq L$ and $1\leq W \leq N_\theta$, $1\leq L\leq c\log 5$ (for $d=2$ and $k=3$), cf. Lemma \ref{lem:tanh-approx}, lead to
\begin{align*}
\mathfrak{R}_n(\mathcal{H}_r)
 \leq&4n^{-\frac12} +12n^{-\frac12}\int^{M_{\mathcal{H}_r}}_{n^{-\frac12}}{\big(cN_\theta \log(4B_\theta \Lambda_{r}\epsilon^{-1})\big)}^{\frac12}\ {\rm d}\epsilon\\
 \leq& 4n^{-\frac12}+12n^{-\frac12}M_{\mathcal{H}_r}\big(cN_\theta \log(4B_\theta\Lambda_r n^{\frac12})\big)^\frac12 \\
 \leq&4n^{-\frac12}+cn^{-\frac12} W^{4L-4}B_\theta^{4L} N_\theta^{\frac12}\big(\log^\frac12 B_\theta+\log^\frac12 \Lambda_{r}+\log^\frac12 n\big)\\
 \leq& c n^{-\frac12} B_\theta^{4L} N_\theta^{4L-\frac{7}{2}} \big( \log^\frac12 B_\theta + \log^\frac12 N_\theta + \log^\frac12 n \big).
\end{align*}
Similarly, repeating the preceding argument leads to
\begin{align*}
 \mathfrak{R}_n(\mathcal{H}_d)&\le c n^{-\frac12} B_\theta^2 N_\theta^{\frac52} (\log^\frac12 B_\theta + \log^\frac12 N_\theta + \log^\frac12 n),\\
 \mathfrak{R}_n(\mathcal{H}_n)&\le c n^{-\frac12} B_\theta^{2L} N_\theta^{2L-\frac32} (\log^\frac12 B_\theta + \log^\frac12 N_\theta + \log^\frac12 n).
 \end{align*}
Finally, the desired result follows from Lemma \ref{lem:PAC}.
\end{proof}

Then combining Lemma \ref{lem:err-decomp} with Lemma \ref{lem:err-stat} yields the following error estimate. Thus, by choosing the numbers $N_r$, $N_d$ and $N_n$ of sampling points sufficiently large, the $L^2(\Omega)$ error of the SEPINN approximation can be made about $O(\epsilon^2)$.
\begin{theorem}\label{thm:err}
Fix a tolerance $\epsilon>0$, and let $(\widehat w_{\theta}, \widehat \gamma)\in \mathcal{W}_\epsilon \times I_\gamma $ be a minimizer to the empirical loss $\widehat{\mathcal{L}}_{\boldsymbol\sigma}(w_\theta,\gamma)$ in \eqref{2deqn:loss-emp}.
Then for any small $\tau$, with the statistical errors  $e_r$, $e_d$ and $e_n$ from Lemma \ref{lem:err-stat}, we have with probability at least  $1- 3 \tau$, the following error estimate holds
\begin{equation*}
\| u^* - \widehat u\|_{L^2(\Omega)}^2 \le c \min(\sigma_d,\sigma_n)^{-1} \big((1+\sigma_d + \sigma_n)\epsilon^2 + e_r + \sigma_d e_d + \sigma_n e_n \big),
\end{equation*}
where the constant $c$ is independent of $\boldsymbol\sigma$.
\end{theorem}

\begin{remark}
According to Theorem \ref{thm:err}, the parameter $\boldsymbol\sigma$ has to be tuned carefully in order to optimally achieve the error: a too small $\boldsymbol{\sigma}$ incurs large consistency errors, while a too large $\boldsymbol{\sigma}$ incurs big statistical errors {\rm(}on approximating the integrals on the boundary{\rm)}.
\end{remark}

\section{Numerical experiments}\label{sec:experiment}
Now we present numerical examples to illustrate  SEPINN,
and compare it with several existing PINN type solvers. Note that our main goal is to illustrate the improved approximation accuracy by singularity enrichment, instead of verifying the convergence rate. Indeed, even for the standard PINN (or any other neural solvers), numerically verifying the theoretical rate remains a daunting challenge \cite{SiegelHong:2023}. In the training, $N_r=10,000$ points in the domain $\Omega$
and $N_b=800$ points on the boundary $\partial\Omega$ are selected  uniformly at random to form the empirical loss $\widehat{\mathcal{L}}_{\boldsymbol{\sigma}}$, unless
otherwise specified. In the path-following (PF) strategy,
we take an increasing factor $q=1.5$. All numerical experiments were conducted on a personal laptop
(Windows 10, with RAM 8.0GB, Intel(R) Core(TM) i7-10510U CPU, 2.3 GHz), with Python 3.9.7, with PyTorch. The gradient of the DNN output $w_\theta(x)$ with respect to the input ${x}$ (i.e.,
spatial derivative) and that of the loss $\widehat{\mathcal{L}}_{\boldsymbol\sigma}$ to $\theta$ are computed via automatic differentiation \cite{Baydin:2018} using \texttt{torch.autograd}.

For SEPINN and SEPINN-C (based on cutoff), we minimize the loss $\widehat{\mathcal{L}}_{\boldsymbol\sigma}$ in two stages: first determine the coefficients $\boldsymbol{\gamma}_N$, and then reduce the boundary
error and refine the DNN approximation $w_\theta$ (of the regular part $w$) by the PF strategy. We use different optimizers at
these two stages. First, we minimize the loss $\widehat{\mathcal{L}}_{\boldsymbol{\sigma}}(w_\theta,\boldsymbol\gamma)$ in both $\theta$ and $\boldsymbol{\gamma}$ using Adam
\cite{KingmaBa:2015} (from the SciPy library), with the default setting (tolerance: 1.0e-8, no
box constraint, maximum iteration number: 1000); then, we minimize $\widehat{\mathcal{L}}_{\boldsymbol\sigma}(w_\theta, \widehat{\boldsymbol\gamma}^*)$ (fixing $\boldsymbol{\gamma}$ at $\widehat{\boldsymbol{\gamma}}^*$ fixed)
using limited memory BFGS (L-BFGS) \cite{ByrdLu:1995}, with the default setting (tolerance: 1.0e-9,
no box constraint, maximum iteration number:  2500). For SEPINN-N, we employ only L-BFGS \cite{ByrdLu:1995}. We have employed different optimizers for SEPINN-C and SEPINN-N: in SEPINN-C, the influence of the DNN parameters $\theta$ and stress intensity factors $\boldsymbol{\gamma}$ on the loss $\widehat{\mathcal{L}}_{\boldsymbol{\sigma}}(w_\theta,\boldsymbol{\gamma})$ differ markedly, and one may use different learning rate for them to compensate the influences. Once the parameter $\boldsymbol{\gamma}$ is fixed, the loss $\widehat{\mathcal{L}}_{\boldsymbol{\sigma}}(w_\theta,\boldsymbol{\gamma})$ can bem efficiently minimized via L-BFGS.
To measure the accuracy of an approximation $\hat w$ of $w^*$, we use the relative
$L^2(\Omega)$-error  $e=\|w^*-\hat w\|_{L^2(\Omega)}/\|w^*\|_{L^2(\Omega)}$, with the error computed using  sampling points in the domain $\Omega$. The stopping
condition of the PF strategy is set to $e<\text{1.00e-3}$ and $\sigma_d^{(k)},\sigma_n^{(k)}\leq \sigma^*$, for
some fixed $\sigma^*>0$. The first condition ensures that $\hat w$ can achieve the desired
accuracy and the second condition terminates the iteration after a finite number of loops. The detailed hyper-parameter setting of the PF strategy is listed in Table \ref{tab:hyper-path}. Throughout the training the box constraint on the DNN parameter is not imposed, since numerically it is observed that the DNN parameters stay bounded during the entire training process. The Python code
for reproducing the numerical experiments will be made available at \url{https://github.com/hhjc-web/SEPINN.git}. The zip file of the complete set of codes is available in the supplementary material.

\begin{table}[hbt!]
\centering
\begin{threeparttable}
\caption{The hyper-parameters for the PF strategy for SEPINN (2D) and SEPINN-C and SEPINN-N (3D) for the examples. The notation $\widehat\gamma^*$ denotes the estimated stress intensity factor, and $e$ the prediction error. }\label{tab:hyper-path}
\begin{tabular}{|c|ccccccc|}
\toprule
Example & $\sigma^*$ & $\boldsymbol{\sigma}^{(1)}$ &  $\boldsymbol{\sigma}^{(K)}$ & $\widehat\gamma^*$ &  epoch &  time(s) & $e$\\
\midrule
 \ref{exam:2d-lshape} & 1200 & 100 & 1139.1 & 1.0001 & 13.4k & 690 & 1.84e-3\\
\midrule
\ref{exam:2d-mix} & 800 &
$\begin{pmatrix}
    100\\100
\end{pmatrix}$ & $\begin{pmatrix}
    759.4\\759.4
\end{pmatrix}$ & 1.0033 & 3.7k & 109 & 4.62e-3\\
\midrule
\ref{exam:3d-lshape} (SE-C) & 4000 & 400 & 3037.5 & Table \ref{table:para} & 8.3k & 4161 & 5.04e-2\\
\ref{exam:3d-lshape} (SE-N) & 5000 & 400 & 4556.3 & -- -- & 13.2k & 2095 & 3.80e-2\\
 \midrule
 \ref{exam:3d-singularities} (SE-C) &  1000 &
 $\begin{pmatrix}
    100\\100
\end{pmatrix}$ & $\begin{pmatrix}
    759.4\\759.4
\end{pmatrix}$ & Fig. \ref{fig:para2} & 10.4k & 10071 & 2.08e-2\\
 \ref{exam:3d-singularities} (SE-N) & 4000 & $\begin{pmatrix}
    100\\400
\end{pmatrix}$ & $\begin{pmatrix}
    759.4\\3037.5
\end{pmatrix}$ & -- -- & 21.2k & 6493 & 3.83e-2\\
\midrule
\ref{exam:eigen} &  600 & 50 & 569.5 & Table \ref{table:eigen} & 15.8k & 1396 & -- --\\
\bottomrule
\end{tabular}
\end{threeparttable}
\end{table}

First, we showcase the approach on an L-shape domain \cite[Example 5.2]{cai2001solution}.
\begin{example}\label{exam:2d-lshape}
The domain $\Omega=(-1,1)^2\backslash \left(\left[0,1\right)\times\left(-1,0\right]\right)$. Set $\rho=1$ and $R=\frac{1}{2}$ in \eqref{cutoff}, the source
	$$ f=\left\{\begin{aligned}
\sin (2 \pi x)\left[2 \pi^2\left(y^2+2 y\right)\left(y^2-1\right)-\left(6 y^2+6 y-1\right)\right]-\Delta\left(\eta_\rho s\right),\,\,-1 \leq y \leq 0,\\
		\sin (2 \pi x)\left[2 \pi^2\left(-y^2+2 y\right)\left(y^2-1\right)-\left(-6 y^2+6 y+1\right)\right]-\Delta\left(\eta_\rho s\right),\,\,0 \leq y \leq 1,
	\end{aligned}\right. $$
with the singular function $s=r^{\frac{2}{3}} \sin \left(\frac{2 \theta}{3}\right)$, and $\Gamma_D=\partial\Omega$. The exact solution $u$ of the problem is given by $u=w+\eta_\rho s$, with the regular part $w$ given by
\begin{equation}
w=\begin{cases}
			\sin (2 \pi x)\left(\frac{1}{2} y^2+y\right)\left(y^2-1\right),&-1 \leq y \leq 0, \\
			\sin (2 \pi x)\left(-\frac{1}{2} y^2+y\right)\left(y^2-1\right),&0 \leq y \leq 1.
		\end{cases}
\end{equation}
\end{example}

\def\figheight {1.85cm}
\begin{figure}[hbt!]
	\centering \setlength{\tabcolsep}{0pt}
	\begin{tabular}{cccccc}
		\includegraphics[height=\figheight]{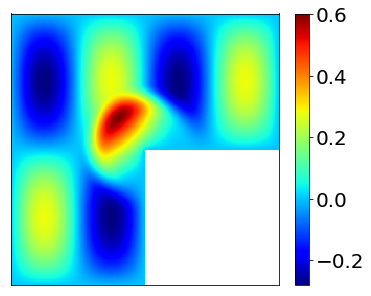}& \includegraphics[height=\figheight]{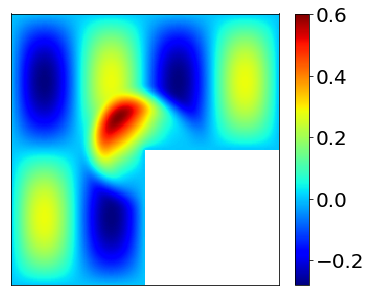}& \includegraphics[height=\figheight]{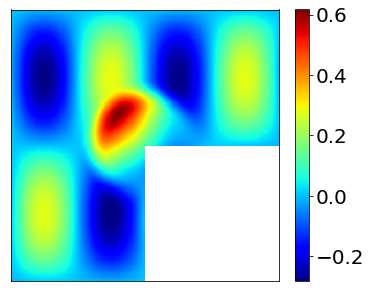}&
		\includegraphics[height=\figheight]{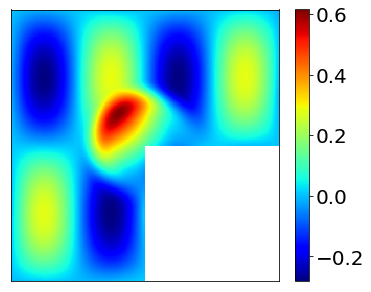}&
		\includegraphics[height=\figheight]{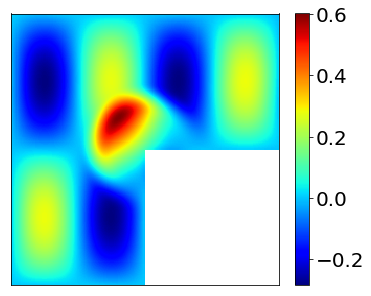}& \includegraphics[height=\figheight]{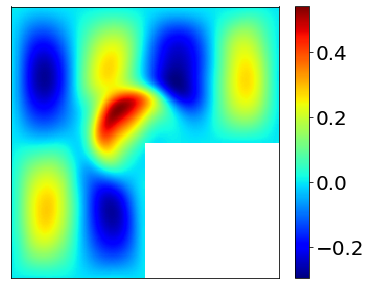}\\
		&\includegraphics[height=\figheight]{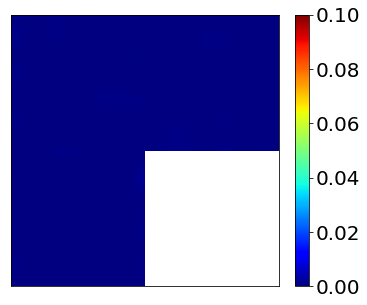} &
		\includegraphics[height=\figheight]{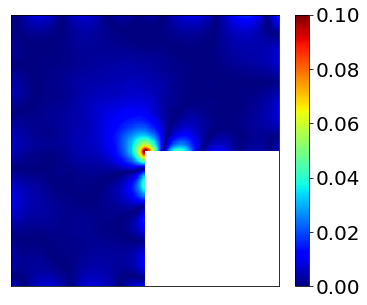} &
		\includegraphics[height=\figheight]{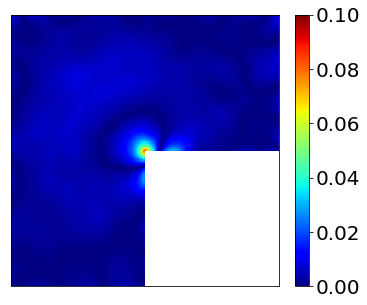} &
		\includegraphics[height=\figheight]{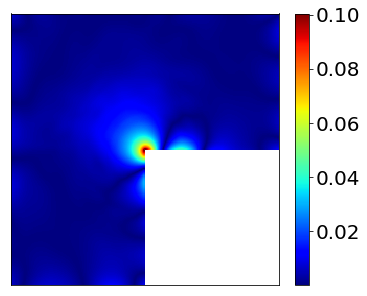} & \includegraphics[height=\figheight]{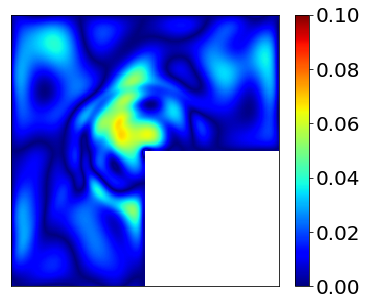}\\
		exact & SEPINN & SAPINN & FIPINN & PINN & DRM
	\end{tabular}
	\caption{The numerical approximations for Example \ref{exam:2d-lshape} by SEPINN, SAPINN, FIPINN, PINN and DRM. From top to bottom: DNN approximation and pointwise error.} \label{fig:exam:2d-lshape}
\end{figure}

\begin{figure}[hbt!]
	\centering\setlength{\tabcolsep}{0pt}
	\begin{tabular}{cc}
		\includegraphics[height=4.8cm]  {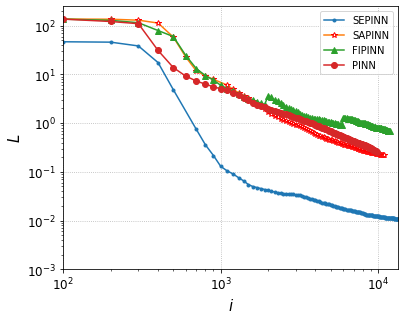} &  \includegraphics[height=4.8cm]  {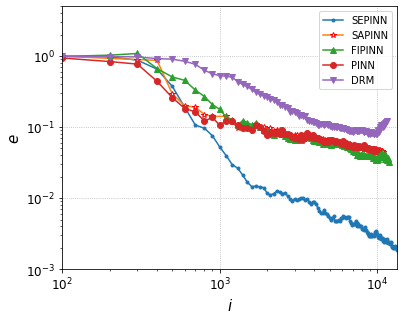}\\
		(a) $\widehat{ \mathcal{L}}$ vs $i$ & (b) $e$ vs $i$
	\end{tabular}
	\caption{\label{fig:exam:2d-lshape compare} Training dynamics for SEPINN and benchmark methods: {\rm(a)} the decay of the empirical loss $\widehat{\mathcal{L}}$ versus the iteration index $i$ {\rm(}counted along the path-following trajectory{\rm)} and {\rm(b)} the error $e$ versus $i$.}
\end{figure}

In this example, the regular part $w$ belongs to $H^2(\Omega)$ but not to $H^3(\Omega)$, and $u$ lies in $H^1(\Omega)$ but not in $H^2(\Omega)$. In SEPINN, we employ a 2-20-20-20-1 DNN (3 hidden layers, each having 20 neurons).
The first stage of the PF strategy gives an estimate $\widehat\gamma^*=1.0001$, and the final prediction error $e$ after the second stage is  $1.84{\rm e}\text{-}3$.
Fig. \ref{fig:exam:2d-lshape} shows that the pointwise error of the SEPINN approximation is  small and the accuracy around the singularity at the reentrant corner is excellent. In contrast, applying PINN and DRM directly fails to yield satisfactory results near the reentrant corner because of the presence of the singular term $r^\frac{2}{3}\sin\frac{2}{3}\theta$, consistent with the approximation theory of DNNs to singular functions \cite{GuhringRaslan:2021}. DRM shows larger errors over the whole domain, not just in the vicinity of the corner. By adaptively adjusting the empirical loss, SAPINN and FIPINN can yield more accurate approximations than PINN, but the error around the singularity is still large. Numerically, FIPINN can adaptively add sampling points near the corner but does not give high concentration, which limits
the accuracy of the final DNN approximation. For all the methods, the largest error occurs near the boundary $\partial\Omega$.

To shed further insights into the methods, we show in Fig. \ref{fig:exam:2d-lshape compare} the training dynamics of the empirical loss $\widehat{\mathcal{L}}$ and relative error $e$, where $i$ denotes the total iteration index along with the PF loops. In Fig. \ref{fig:exam:2d-lshape compare}, we have omitted the loss curve for DRM, since its value is negative. For all methods, the loss $\widehat{\mathcal{L}}$ and error $e$ both decay steadily as the iteration proceeds, indicating stable convergence of the optimizer, but SEPINN enjoys the fastest decay and smallest error $e$, due to the improved regularity of $w^*$. The final error $e$ saturates at around $ 10^{-3} $ for SEPINN and $10^{-2}$ for PINN, SAPINN and FIPINN, but only $10^{-1}$ for DRM. Indeed the accuracy of neural PDE solvers tends to stagnate at a level of $10^{-2}\sim 10^{-3}$ \cite{RAISSI2019686,yu2018deep,Zang:2020,CuomoSchiano:2022}.

We next investigate a mixed boundary value problem \cite[Example 1]{cai2006finite}.
\begin{figure}[hbt!]
	\centering \setlength{\tabcolsep}{0pt}
	\begin{tabular}{cccccc}
		\includegraphics[height=\figheight]{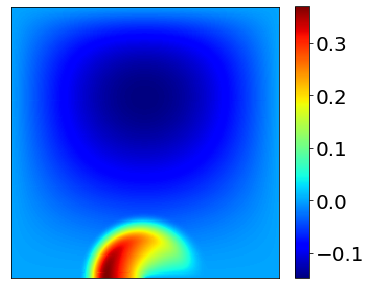}&
        \includegraphics[height=\figheight]{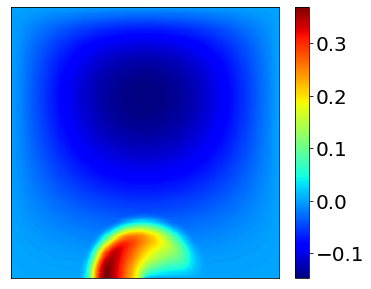}& \includegraphics[height=\figheight]{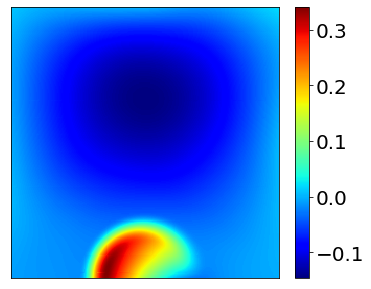}&
		\includegraphics[height=\figheight]{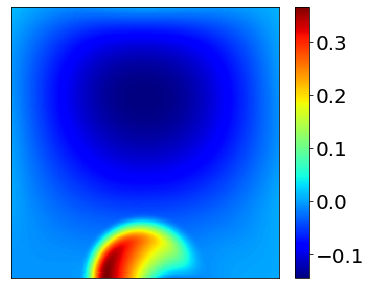}&
		\includegraphics[height=\figheight]{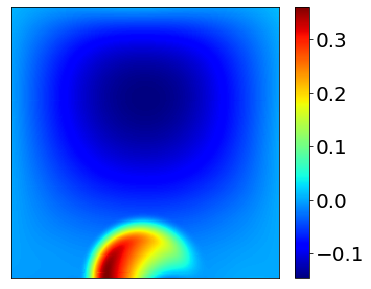}&
            \includegraphics[height=\figheight]{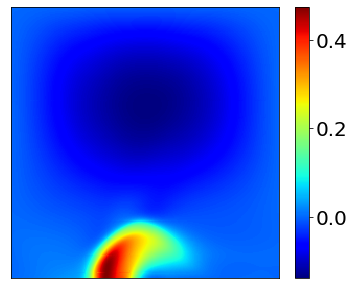}\\
		&\includegraphics[height=\figheight]{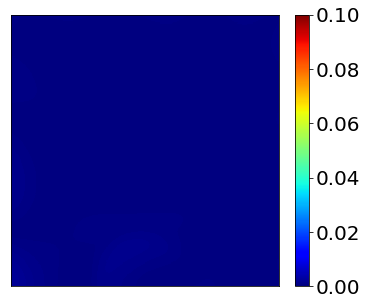} &
		\includegraphics[height=\figheight]{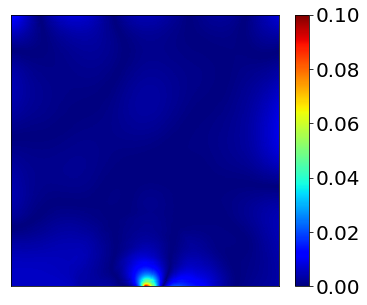} &
		\includegraphics[height=\figheight]{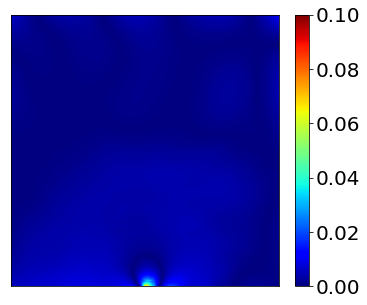} &
		\includegraphics[height=\figheight]{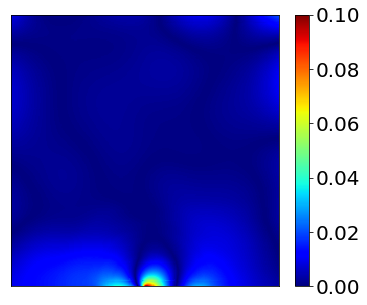}&
            \includegraphics[height=\figheight]{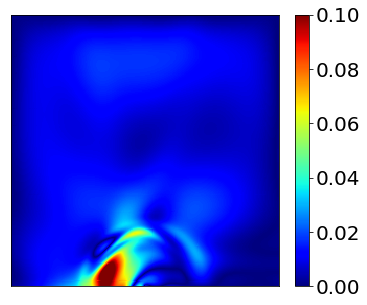}\\
	exact	& SEPINN & SAPINN & FIPINN & PINN & DRM
	\end{tabular}
	\caption{The numerical approximations of Example \ref{exam:2d-mix} by the proposed SEPINN {\rm(}error: $\text{4.62e-3}${\rm)}, SAPINN {\rm(}error: $\text{4.40e-2}${\rm)}, FIPINN {\rm(}error: $\text{3.65e-2}${\rm)}, PINN {\rm(}error: $\text{7.33e-2}${\rm)} and DRM {\rm(}error: $\text{1.86e-1}${\rm)}. From top to bottom: DNN approximation and pointwise error.} \label{fig:exam:2d-mix}
\end{figure}

\begin{example}\label{exam:2d-mix}
The domain $\Omega$ is the unit square $\Omega=(0,1)^2$, $\Gamma_N=\{(x,0):x\in(0,\frac{1}{2})\}$ and $\Gamma_D=\partial\Omega\backslash \Gamma_N$. The singular function
$s=r^{\frac{1}{2}}\sin\frac{\theta}{2}$ in the local polar coordinate $ (r,\theta) $ at $ (\frac{1}{2}, 0) $. Set $\rho=1$ and $R=\frac{1}{4}$ in \eqref{cutoff}, the source $f=-\sin(\pi x)(-\pi^2y^2(y-1)+6y-2)-\Delta(\eta_\rho s).$ The exact solution $u=\sin(\pi x)y^2(y-1)+\eta_\rho s$, with the regular part $w=\sin(\pi x)y^2(y-1)$ and stress intensity factor $\gamma=1$.
\end{example}

This problem has a geometric singularity at the point $ (\frac{1}{2}, 0) $, where the boundary condition changes from Dirichlet to Neumann with an interior angle $ \omega = \pi $. We employ a 2-10-10-10-1 DNN (with 3 hidden layers, each having 10 neurons). The first stage of the PF strategy gives an estimate $\widehat\gamma^*=1.0033$, and the prediction error $e$ after the second stage is 4.62e-3. The singularity at the crack point is accurately resolved by SEPINN, cf. Fig. \ref{fig:exam:2d-mix}, which shows also the approximations by DRM and other PINN techniques. The overall solution accuracy is very similar to Example \ref{exam:2d-lshape}, and SEPINN achieves the smallest error. SAPINN and FIPINN improve the standard PINN, but still suffer from large errors near the singularity. Like before, the maximum error occurs near the boundary $\partial\Omega$, especially near the singularity point.

In addition, we have experimented with different architectures (depth, width) of the neural networks for the example. The training dynamics of the loss $\widehat{\mathcal{L}}$ and the relative error $e$ with four different NNs are shown in Fig. \ref{fig:exam:2d-mix-net}. The loss $\widehat{\mathcal{L}}$ for the 2-20-20-20-20-20-1 network cannot stabilize at a lower level than the 2-20-20-20-1 network. Thus, increasing the size of the NN or the number of sampling points alone does not necessarily lead to a smaller error for the approximation, as observed earlier in \cite{JinLiLu:2022}. It remains to develop practical guidelines to select suitable architecture balancing good accuracy and computational expense.

\begin{figure}[hbt!]
\centering\setlength{\tabcolsep}{0pt}
\begin{tabular}{cc}
    \includegraphics[height=5.5cm]  {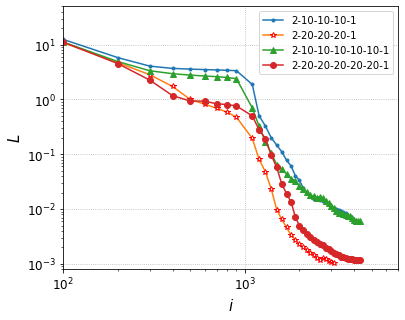} &
    \includegraphics[height=5.5cm]  {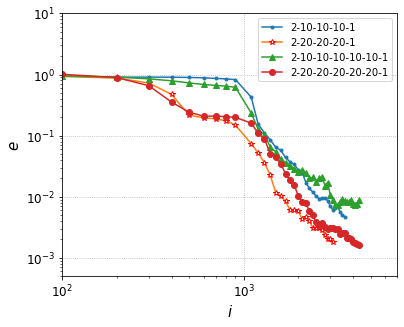} \\
    (a) $L$ vs $i$ & (b) $e$ vs $i$
\end{tabular}
 \caption{\label{fig:exam:2d-mix-net} The training dynamics for Example \ref{exam:2d-mix} with four different NNs: {\rm(a)} the loss $L$ versus the iteration index $i$ and {\rm(b)} the error $e$ versus the iteration index $i$.}
\end{figure}

The next example is a 3D Poisson equation adapted from \cite[Example 1]{nkemzi2021singular}.
\begin{example}\label{exam:3d-lshape}
Let $\Omega_0=(-1,1)^2\backslash\ \left(\left[0,1\right)\times\left(-1,0\right]\right)$, and the domain $\Omega=\Omega_0\times(-1,1)$. Define
$\Phi(r,z)=-2\arctan\dfrac{{\rm e}^{-\pi r}\sin \pi z}{1+{\rm e}^{-\pi r}\cos \pi z}=2\sum_{n=1}^\infty \dfrac{(-1)^n{\rm e}^{-n\pi r}\sin n\pi z}{n}$,
and set $\rho=1$ and $R=\frac{1}{2}$ in \eqref{cutoff}, the source $f=6x(y-y^3)(1-z^2) + 6y(x-x^3)(1-z^2) + 2(y-y^3)
(x-x^3)-\Delta(\Phi\eta_\rho s)$, with the singular function $s=r^{\frac{2}{3}} \sin (\frac{2 \theta}{3})$, and a zero
Dirichlet boundary condition.  The exact solution $u$ is given by $u=(x-x^3)(y-y^3)(1-z^2)+\Phi\eta_\rho s$.
\end{example}

The coefficients $\gamma_n^*$ are given by $\gamma_0^*=0$ and $\gamma_n^*=(-1)^n
\frac{2}{n}$ for $n\in\mathbb{N}$. In SEPINN-C (cutoff), we take a truncation level $N=20$ to approximate the first $N+1$
coefficients in the series for $\Phi(r,z)$, and a  3-10-10-10-1 DNN to approximate $w$.
In the first stage, we employ a learning rate $1.0{\rm e}\text{-}3$ for the DNN parameters $\theta$ and $8.0{\rm e}\text{-}3$ for
coefficients $\boldsymbol\gamma_N$, and the prediction error $e$ after the second stage is 5.04e-2. We present the slices at $z=\frac{1}
{2}$ and $ z=\frac{1}{4} $ in Fig. \ref{fig:exam:3d-lshape}. The true and estimated values of  $\boldsymbol{\gamma}_n$ are given in Table
\ref{table:para}: the first few terms of the expansion are well approximated, but
as the index $n$ gets larger, the approximations of $\widehat{\gamma}_n^*$ becomes less accurate. However, the precise mechanism for the observation remains unclear.  To shed insights, in Table \ref{table:converging rate}, we show the relative errors $e$ of $\hat w$, $e_S$ of the singular part $\widehat S = \widehat{\gamma_0}\eta_\rho s + \sum_{n=1}^N \widehat{\gamma_n}{\rm e}^{-\xi_n r}Z_n(z)\eta_\rho s$, and $e_u$ of the approximation $\hat u=\hat{w}+\widehat S$, defined by $e_u=\|u^*-\hat u\|_{L^2(\Omega)}/\|u^*\|_{L^2(\Omega)}$, and likewise for $e_S$. The absolute errors are also given in order to give the full picture. Note that $e$ and $e_u$ exhibit very similar behavior, and both tend to stagnate, which is expected due to the slow decay of the coefficients $\gamma_n^*$. The error $e_S$ decays steadily as $N$ increases. Thus, the non-steady convergence of $e_u$ is due to the approximation of $w$, which might be due to the optimization error during the training.

Next we present the SEPINN-N approximation. We use two 4-layer DNNs, both of 3-10-10-10-1, for $w$ and $\Phi$. The PF strategy is run with a maximum 2500 iterations for each fixed $\boldsymbol{\sigma}$, and the final prediction error $e$ is 3.80e-2.
The maximum error of the SEPINN-N approximation is slightly smaller, and both can give excellent approximations.

\begin{table}[hbt!]
\setlength{\tabcolsep}{3pt}
\centering
\begin{threeparttable}
    \caption{ The estimated values of the parameters $\gamma_n$ for Example \ref{exam:3d-lshape}, with five significant digits. \label{table:para}}
	\centering
		\begin{tabular}{c|ccccccc}
			\toprule

			 & $\gamma_0$ & $\gamma_1$ & $\gamma_2$ & $\gamma_3$ & $\gamma_4$ & $\gamma_5$ & $\gamma_6$ \\
			\hline
			exact & 0.000{\rm e}0 & -2.000{\rm e}0 & 1.000{\rm e}0 &  -6.667{\rm e}-1  &  5.000{\rm e}-1  & -4.000{\rm e}-1 & 3.333{\rm e}-1\\
			\hline
			predicted & -7.793{\rm e}-5 & -2.000{\rm e}0 & 1.006{\rm e}0 & -6.578{\rm e}-1 & 5.135{\rm e}-1 & -4.238{\rm e}-1 & 3.271{\rm e}-1 \\
			\midrule
			 & $\gamma_7$ & $\gamma_8$ & $\gamma_9$ & $\gamma_{10}$ & $\gamma_{11}$ & $\gamma_{12}$ & $\gamma_{13}$\\
			\hline
			exact  & -2.857{\rm e}-1 & 2.500{\rm e}-1 &  -2.222{\rm e}-1  & 2.000{\rm e}-1  & -1.818{\rm e}-1 & 1.667{\rm e}-1 & -1.538{\rm e}-1\\
			\hline
			predicted& -3.612{\rm e}-1 & 2.410{\rm e}-1 & -2.384{\rm e}-1 & 2.724{\rm e}-1 & -1.855{\rm e}-1& 3.700e-1 & -1.045{\rm e}-2\\
			\midrule

			  & $\gamma_{14}$ & $\gamma_{15}$ & $\gamma_{16}$ & $\gamma_{17}$ & $\gamma_{18}$ & $\gamma_{19}$ & $\gamma_{20}$\\
			\hline
			exact  & 1.429{\rm e}-1 & -1.333{\rm e}-1  &  1.250{\rm e}-1  & -1.176{\rm e}-1 & 1.111{\rm e}-1 & -1.053{\rm e}-1 & 1.000{\rm e}-1\\
			\hline
			predicted  & 2.761{\rm e}-1 & -2.082{\rm e}-1 & 2.867{\rm e}-1 & 1.239{\rm e}-3& 1.961{\rm e}-1 & 7.774{\rm e}-2 & 4.955{\rm e}-1\\
	        \bottomrule
		\end{tabular}
\end{threeparttable}
\end{table}

\begin{table}[hbt!]
	\centering
 \begin{threeparttable}
     \caption{\label{table:converging rate} Convergence of the SEPINN-C approximation with respect to the truncation level $N$ for Example \ref{exam:3d-lshape}.}
	\begin{tabular}{c|cccc}
		\toprule
		$N$ & 5 & 10 & 15 & 20\\
        \midrule
        $e_{abs}$ & 1.76 & 0.797 & 1.42 & 0.937 \\
        \hline
		  $e$ & 7.14{\rm e}-2 & 3.23{\rm e}-2 & 5.77{\rm e}-2 & 3.80{\rm e}-2\\
        \midrule
        $e_{u,abs}$ & 1.78 & 0.835 & 1.44 & 0.951 \\
        \hline
        $e_u$ & 5.12{\rm e}-2 & 2.40{\rm e}-2 & 4.17{\rm e}-2 & 2.73{\rm e}-2\\
        \midrule
        $e_{S,abs}$ & 0.625 & 0.271 & 0.245 & 0.146 \\
        \hline
        $e_S$ & 2.55{\rm e}-2 & 1.11{\rm e}-2 & 1.00{\rm e}-2 & 5.95{\rm e}-3\\
		\bottomrule
  	\end{tabular}
 \end{threeparttable}
\end{table}

\begin{figure}[hbt!]
\centering\setlength{\tabcolsep}{0pt}
	\begin{tabular}{ccccc}
		\includegraphics[height=1.9cm]  {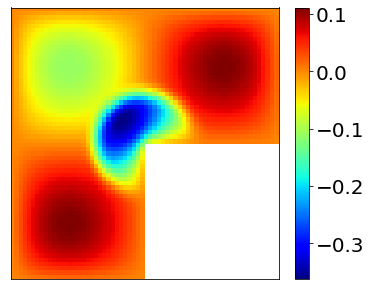} & \includegraphics[height=1.9cm]  {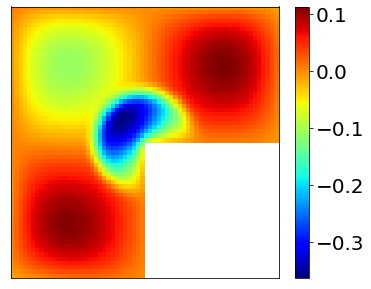} & \includegraphics[height=1.9cm]  {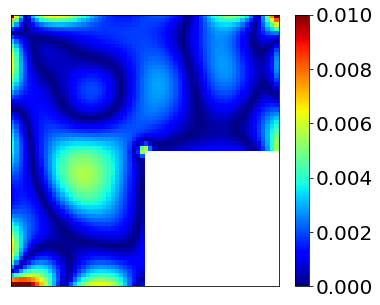} & \includegraphics[height=1.9cm]  {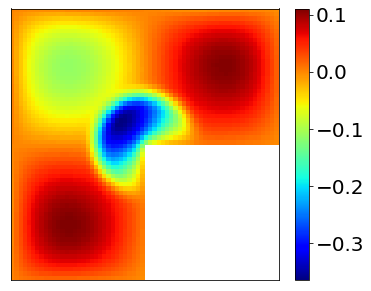} & \includegraphics[height=1.9cm]  {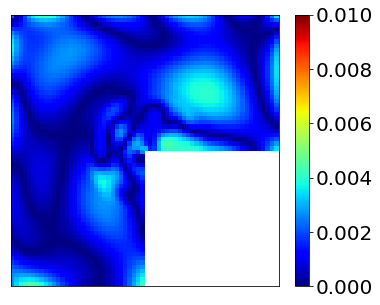}\\
		\includegraphics[height=1.9cm]  {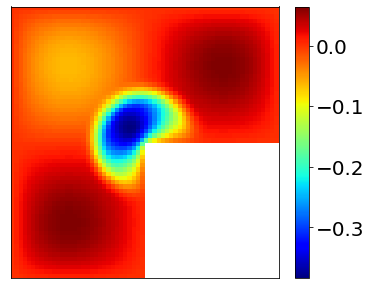} & \includegraphics[height=1.9cm]  {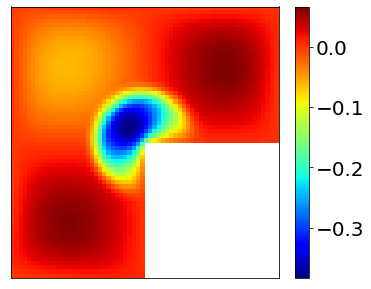} & \includegraphics[height=1.9cm]  {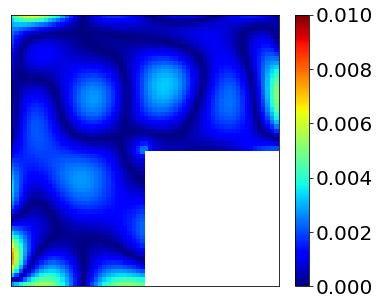} & \includegraphics[height=1.9cm]  {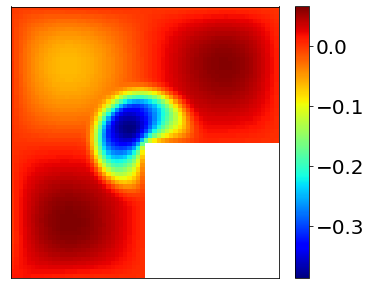} & \includegraphics[height=1.9cm]  {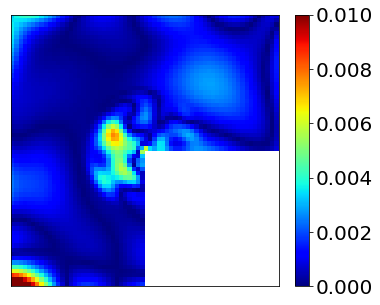}\\
		{\tiny(a)  exact} & {\tiny (b) SEPINN-C}  & {\tiny(c) error, SEPINN-C }& {\tiny (d) SEPINN-N} & {\tiny (e) error, SEPINN-N}
	\end{tabular}
	\caption{\label{fig:exam:3d-lshape} The SEPINN-C and SEPINN-N approximations for Example \ref{exam:3d-lshape},  slices at $z=\frac{1}{2}$ $($top$)$ and $z=\frac{1}{4}$ $($bottom$)$.}
\end{figure}

Fig. \ref{fig:exam:3d-lshape compare} compares the training dynamics for SEPINN-C and SEPINN-N. Fig. \ref{fig:exam:3d-lshape compare} (a) shows the convergence for the fist few flux intensity factors $\boldsymbol\gamma$,  all initialized to 1, which are far from the optimal values. Nonetheless, the algorithm converges to the optimal one steadily. Within a few hundred of iterations, the iterates approximate the exact one well. To accurately approximate $w$, more iterations are needed. Fig. \ref{fig:exam:3d-lshape compare} (b) shows the training dynamics for the DNN $\Phi_\zeta$, where $e_\Phi=\|\Phi^*-\Phi_\zeta\|_{L^2(G)}/\|\Phi^*\|_{L^2(G)}$ ($G$ is the support of the cut-off function $\eta_\rho$ and $\Phi^*$ is the exact one). We study the $L^2(G)$ error instead of the $L^2(\Omega)$ error, since $\eta_\rho$ localizes its influence to $G$. The error $e_\Phi$ eventually decreases to $ 10^{-2} $. The entire training process of the two methods is similar, cf.  Figs. \ref{fig:exam:3d-lshape compare} (c) and (d). SEPINN-N takes more iterations than SEPINN-C, but SEPINN-C training takes longer: SEPINN-C requires evaluating the coefficient $\gamma_n$, which incurs taking Laplacian of singular terms, whereas in SEPINN-N, all parameters are trained together.

\begin{figure}[hbt]
	\centering\setlength{\tabcolsep}{0pt}
	\begin{tabular}{cccc}
		\includegraphics[width=0.52\textwidth]  {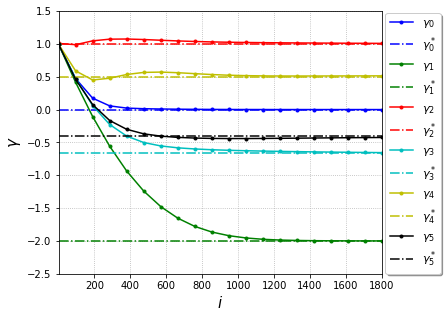} &
		\includegraphics[width=0.47\textwidth]  {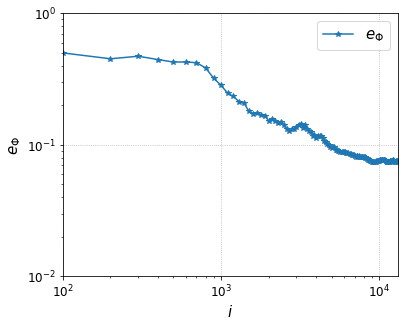} \\
       (a) $\gamma$ vs $i$ & (b) $e_\Phi$ vs $i$ \\ \includegraphics[width=0.47\textwidth]  {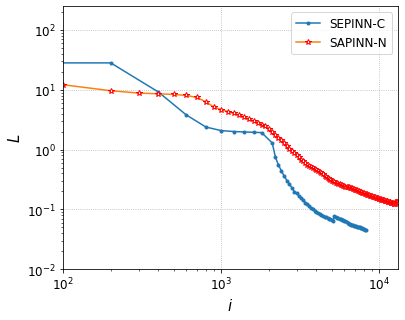} &
	    \includegraphics[width=0.47\textwidth]  {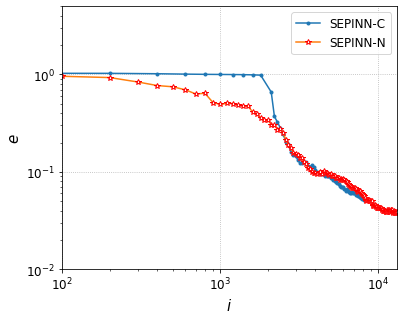}\\		
        (c) $\widehat{\mathcal{L}}$ vs $i$ & (d) $e$ vs $i$
	\end{tabular}
	\caption{\label{fig:exam:3d-lshape compare} The training dynamics of SEPINN-C and SEPINN-N: {\rm(a)} the variation of first few coefficients versus interation index $i$, {\rm(b)} the error $e_{\Phi}$ versus iteration index $i$, {\rm(c)} the decay of the loss $\widehat{\mathcal{L}}$ versus iteration index $i$, {\rm(d)} the error $e$ versus iteration index $i$.}
\end{figure}

The next example involves a combination of four singularities.
\begin{example}\label{exam:3d-singularities}
Let the domain $\Omega=(-\pi,\pi)^3$, $\Gamma_D=\{(x,-\pi,z):x\in(-\pi,0),z\in(-\pi,\pi)\}\cup\{(-\pi,y,z):y\in(-\pi,0),z\in(-\pi,\pi)\}\cup\{(x,\pi,z):x\in(0,\pi),z\in(-\pi,\pi)\}\cup\{(\pi,y,z):y\in(0,\pi),z\in(-\pi,\pi)\}$ and $\Gamma_D=\partial\Omega\backslash \Gamma_N$. Let the vertices $ \boldsymbol{v}_1:(0,-\pi) $, $ \boldsymbol{v}_2:(\pi,0) $, $ \boldsymbol{v}_3:(0,\pi) $ and $ \boldsymbol{v}_4:(-\pi,0) $. This problem has four geometric singularities at boundary edges $ \boldsymbol{v}_j\times(-\pi,\pi) $, where the type of the boundary condition changes from Dirichlet to Neumann with interior angles $ \omega_j =\pi,$ $ j=1,2,3,4$. Set
$ \Phi_j(r_j,z)=r-\ln(2\cosh r_j-2\cos z)=\sum_{n=1}^\infty \dfrac{2}{n}{\rm e}^{-nr_j}\cos nz$, $ j=1,2,3,4$,
and set $\rho_j=1$ and $R=\frac{1}{2}$ in \eqref{cutoff}, the source
    $	f=((\frac{4}{\pi^2}-\frac{12x^2}{\pi^4})(1-\frac{y^2}{\pi^2})^2 + (\frac{4}{\pi^2}-\frac{12y^2}{\pi^4})(1-\frac{x^2}{\pi^2})^2 + (1-\frac{x^2}{\pi^2})^2(1-\frac{y^2}{\pi^2})^2)\cos z-\sum_{j=1}^4\Delta(\Phi_j\eta_{\rho_j} s_j)$
with the singular functions $s_j=r_j^{\frac{1}{2}} \cos(\frac{ \theta_j}{2})$ for $j=1,3$ and $s_j=r_j^{\frac{1}{2}} \sin(\frac{ \theta_j}{2})$ for $j=2,4$. The exact solution $u$ is given by
   $ u=(1-\frac{x^2}{\pi^2})^2(1-\frac{y^2}{\pi^2})^2\cos z+\sum_{j=1}^4\Phi_j\eta_{\rho_j} s_j$.
\end{example}

\newcommand{\fight}{2.2cm}
\begin{figure}[hbt!]
	\centering\setlength{\tabcolsep}{0pt}
	\begin{tabular}{ccccc}
		\includegraphics[height=\fight]  {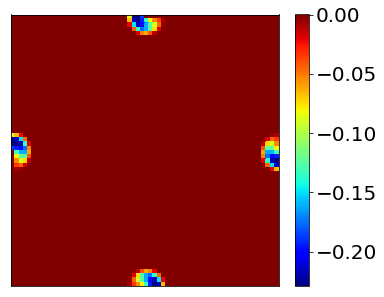} & \includegraphics[height=\fight]  {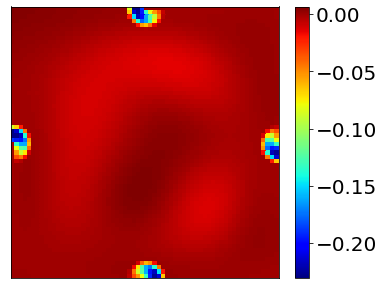} & \includegraphics[height=\fight]  {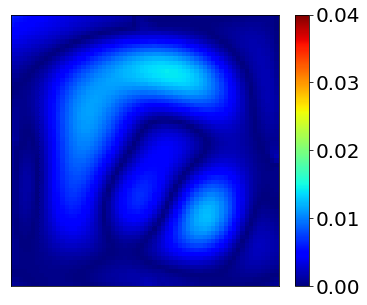} & \includegraphics[height=\fight]  {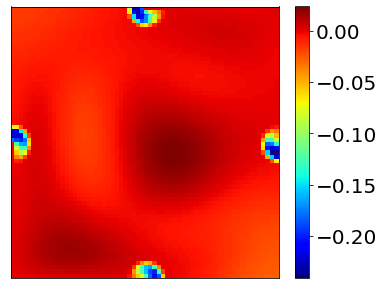} & \includegraphics[height=\fight]  {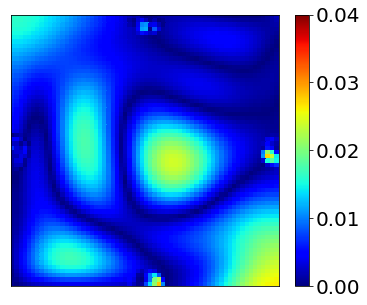}\\
		\includegraphics[height=\fight]  {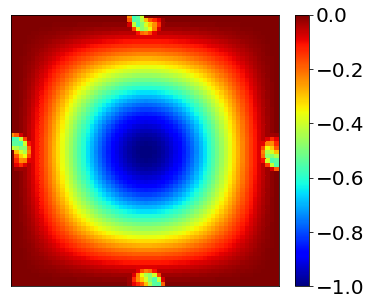} & \includegraphics[height=\fight]  {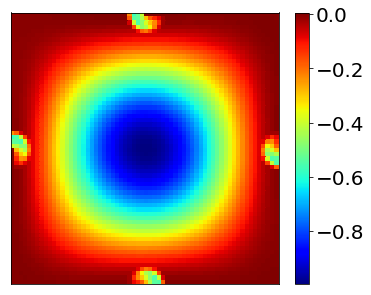} & \includegraphics[height=\fight]  {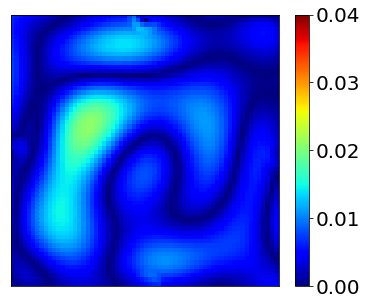} & \includegraphics[height=\fight]  {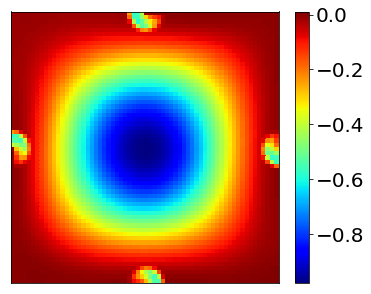} & \includegraphics[height=\fight]  {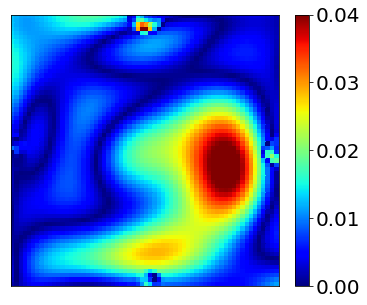}\\
		{\tiny(a)  exact} &{\tiny (b)  SEPINN-C}  &{\tiny (c)  error, SEPINN-C} &{\tiny  (d) SEPINN-N} &{\tiny (e)  error, SEPINN-N}
	\end{tabular}
	\caption{\label{fig:exam:3d-singularities} The SEPINN-C and SEPINN-N approximations for Example \ref{exam:3d-singularities},   slices at $z=\frac{\pi}{2}$ {\rm(}top{\rm)} and $z=\pi$ {\rm(}bottom{\rm)}.}
\end{figure}

This example requires learning more  parameters regardless of the method: for SEPINN-C, we have to expand  four singular functions and learn their coefficients, whereas for SEPINN-N, we employ five networks to approximate $w$ and $\Phi_{\zeta,j}\ (j=1,2,3,4)$. We take the number of sampling points $N_d=800$ and $N_n=1200$ on the boundaries $\Gamma_D$ and $\Gamma_N$, respectively, and $N_r=10000$ in the domain $\Omega$.

First we present the SEPINN-C approximation. Note that $\gamma_0^*=0$ and $\gamma_n^*=\frac{2}{n}$ for $n\in\mathbb{N}$. In SEPINN-C, we take a truncation level $N=15$ for all four singularities.
In the first stage, we take a learning rate $2.0\text{e-}3$ for $\theta$ and for coefficients $r_1=r_3=1.1\text{e-}2$, $r_2=r_4=7.0\text{e-}3$. The final prediction error $e$ is 2.08e-2, and the estimated coefficients $\widehat{\boldsymbol\gamma}^*$ are shown in Fig. \ref{fig:para2}. The first few coefficients are well approximated, but the high-order ones are less accurate.

\begin{figure}[hbt!]
    \centering
	\includegraphics[width=.5\textwidth]{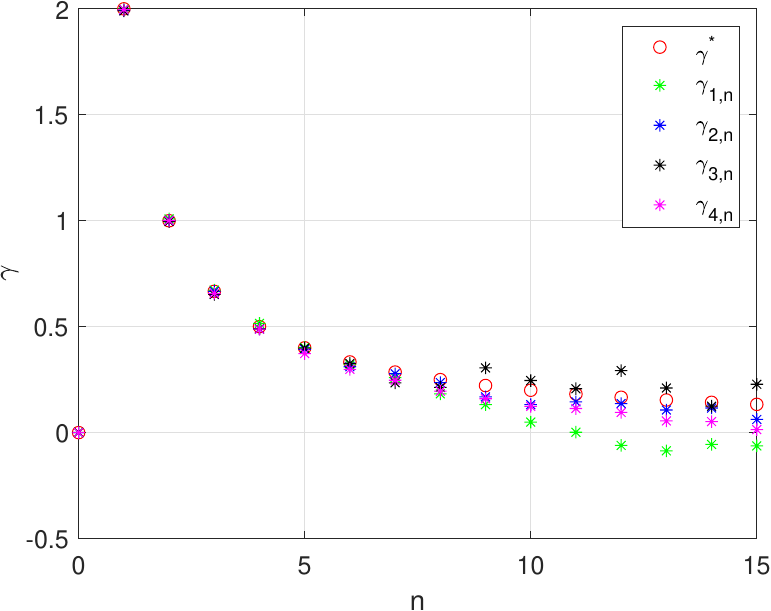}
	\caption{\label{fig:para2} The comparison between true and estimated values of the parameters $\gamma_{j,n}$ for Example \ref{exam:3d-singularities}.}
\end{figure}

Next we show the SEPINN-N approximation, obtained with five 4-layer 3-10-10-10-1 DNNs to approximate $w$ and $\Phi_{\zeta,j}\ (j=1,2,3,4)$ separately. The training process suffers from the following problem: using only L-BFGS tends to be trapped into a local minimum of the loss $\widehat{\mathcal{L}}_{\boldsymbol\sigma}$, which persists even after adjusting extensively the hyper-parameters. Therefore, we first train the DNNs with Adam (learning rate $r=4.0\text{e-}3$) for 1000 iterations and then switch to L-BFGS (learning rate $r=0.2$) for a maximum 4000 iterations. This training strategy greatly improves the accuracy. The prediction error $e$ after the second stage is 3.83e-2. The approximation is fairly accurate but slightly less accurate than that by SEPINN-C in Fig. \ref{fig:exam:3d-singularities} in both $L^\infty(\Omega)$ and $L^2(\Omega)$ norms.

Last, we illustrate SEPINN on the Laplacian eigenvalue problem.
\begin{example}\label{exam:eigen}
Let $\Omega=(-1,1)^2\backslash(\left[0,1\right)\times\left(-1,0\right])$. Consider the following Dirichlet Laplace eigenvalue problem:
$-\Delta u = \mu u$, in $\Omega$ and $u=0$ on $ \partial\Omega$,
where $ \mu >0$ is the eigenvalue and $u\not\equiv0$ is the corresponding eigenfunction.
\end{example}

The eigenvalue problem on an  L-shaped domain has been studied extensively \cite{doi:10.1137/0704008,doi:10.1137/120878446,YUAN20091083}. The eigenfunctions
may have singularity around the reentrant corner, but the analytic forms appear unavailable: the first eigenfunction $ u_1$ has a leading singular term $ r^{\frac{2}{3}}\sin(\frac{2}{3}\theta)$, the second
one $ u_2 $ has $ r^{\frac{4}{3}}\sin(\frac{4}{3}\theta)$ \cite{doi:10.1137/0704008}, and the third $u_3$ is analytic, given by  $ u_3(x_1,x_2) =\sin(\pi x_1)\sin(\pi x_2) $. We compute the first two leading eigenpairs $(\mu_1,u_1)$ and $(\mu_2,u_2)$. The preceding discussions indicate $u_1\in H^1(\Omega)$ and $u_2\in H^2(\Omega)$, and we split the leading singularity from both functions in order to benefit from SEPINN, i.e.,
$u_i=w_i+\gamma_i\eta_\rho s$,
with $ s=r^{\frac{2}{3}}\sin(\frac{2\theta}{3}) $ and $w_i$ is approximated by a DNN.
Following the ideas in \cite{10.1162/necoa01583} and SEPINN, we employ the following loss
\begin{align}
	\mathcal{L}_{\boldsymbol\sigma}(w_1,&w_2;\gamma_1,\gamma_2)
	=\sum_{i=1}^{2}\big(\|\Delta (w_i+\gamma_i\eta_\rho s)+\psi(u_i)(w_i+\gamma_i\eta_\rho s)\|_{L^2(\Omega)}^2+\sigma_1\|w_i\|_{L^2(\partial\Omega)}^2\nonumber\\
   &+\alpha\big|\|w_i+\gamma_i\eta_\rho s\|_{L^2(\Omega)}^2-1 \big |+\nu_i\psi(u_i)\big)
   +\beta\left|(w_1+\gamma_1\eta_\rho s,w_2+\gamma_2\eta_\rho s)_{L^2(\Omega)}\right|.\nonumber
\end{align}
where $\alpha$, $\nu_i(i=1,2)$ and $\beta$ are hyper-parameters and $\psi(u_i)$ is the Rayleigh quotient:
\begin{equation}\label{rayleigh}
\psi(u_i)={\|\nabla u_i\|^2_{L^2(\Omega)}} / {\|u_i\|^2_{L^2(\Omega)}},\quad i=1,2,
\end{equation}
which estimates the eigenvalue $\mu_i$ using the eigenfunction $u_i$, by Rayleigh's principle. We employ an alternating iteration method: we first approximate the eigenfunction $u_i$ by minimizing the loss $\mathcal{L}_\sigma$ and then update the eigenvalue $\mu_i$ by \eqref{rayleigh}, which is then substituted back into $\mathcal{L}_\sigma$. These two steps are repeated until convergence.

In SEPINN, we employ two 2-10-10-10-10-10-10-1 DNNs to approximate the regular parts $w_1$ and $w_2$, and take $\alpha=100$, $\beta=135$, $\nu_1=0.02$ and $\nu_2=0.01$, and determine the parameter $\sigma_d$ by the PF strategy. We use Adam with a learning rate $ 2\text{e-}3 $ for all DNN parameters and $\boldsymbol\gamma$.
Table \ref{table:eigen} shows that singularity enrichment helps solve the eigenvalue problem. Indeed, we can approximate $u_1$ better and get more accurate eigenvalue estimates. Note that during the training process of the standard PINN, the DNN approximation actually directly approaches $u_2$ and cannot capture $u_1$, since it cannot be resolved accurately. Even with a larger penalty $\nu_1$ for the first eigenvalue, this does not improve the accuracy, while $u_2$ can be approximated well due to the better regularity of $u_2$, i.e., $u_2\in H^2(\Omega)$.

\begin{table}[H]
\centering
\begin{threeparttable}
\caption{\label{table:eigen} The estimated eigenvalues by PINN and SEPINN for Example \ref{exam:eigen}.}
	\centering
	\begin{tabular}{ccccc}
		\toprule
		eigenvalue & reference \cite{doi:10.1137/0704008} & SEPINN & stress intensity factor & PINN \\
		\midrule
		$\mu_1$ & 9.6397 & 9.6226 & $\gamma_1$=1.82{\rm e}-2 & 10.4840\\
		\hline
		$\mu_2$ & 15.1973 & 15.2100 & $\gamma_2$=8.71{\rm e}-3 & 15.0462\\
		\bottomrule
	\end{tabular}
\end{threeparttable}
\end{table}

\section{Conclusions}\label{sec:concl}
In this work, we have developed a family of novel neural solvers for second-order elliptic boundary value problems with geometric singularities, e.g., corner singularity and mixed boundary conditions in the 2D case, and edge singularities in the 3D case. The basic idea is to enrich the ansatz space spanned by deep neural networks, by incorporating suitable singular functions. These singular functions are designed to capture  leading singularities of the exact solution. In so doing, we can obtain more accurate approximations.  We discuss several variants of the method, depending on the specific scenarios, and discuss the extensions to the eigenvalue problem. Additionally, we provide preliminary theoretical guarantees of the approach.  Numerical experiments indicate that the approach is indeed highly effective and flexible, and works for a broad range of problem settings.

There are several directions for further research. First, one interesting question is to extend the approach to more general singularities, e.g., conic singularities, singularities along curved edges, rotating solids $\Omega=\Omega_0\times(0,l)$ (with $\Omega_0$ depending on the height) and higher-dimensional cases, as well as more complex problems, e.g., Maxwell system and linear elasticity system. The key is to identify the singularity functions. Note that in principle, the singularity functions can be approximate, e.g., computed using FEM. Second, it is of importance to study more closely the optimization aspects of the approach. It is already challenging for standard neural solvers \cite{Krishnapriyan:2021,WangPerdikaris:2022jcp}, the issue appears even more delicate for SEPINN, as indicated by the numerical experiments. Third, it is of much interest to design techniques that allow imposing the boundary conditions exactly, which would facilitate deriving stronger and better estimates.

\section*{Acknowledgements} The authors are grateful to two anonymous referees for their many constructive comments which have led to a significant improvement of the quality of the paper.

\appendix

\section{Technical preliminaries}
In this appendix, we collect several preliminary results that are used in the error analysis of the proposed SEPINN. The Rademacher complexity \cite{AnthonyBartlett:1999,BartlettMendelson:2002} measures the complexity of a collection of functions
by the correlation between function values with Rademacher random variables, i.e., with probability $P(\omega=1)=P(\omega=-1)=\frac12$.
\begin{definition}\label{def: Rademacher}
Let $\mathcal{F}$ be a real-valued function class defined on the domain $D$ and $\xi=\{\xi_j\}_{j=1}^n$ be i.i.d. samples from the distribution $\mathcal{U}(D)$.
Then the Rademacher complexity $\mathfrak{R}_n(\mathcal{F})$ of $\mathcal{F}$ is defined by
\begin{equation*}
\mathfrak{R}_n(\mathcal{F})=\mathbb{E}_{\xi,\omega}\bigg{[}\sup_{f\in\mathcal{F}}\ n^{-1}\bigg{\lvert}\ \sum_{j=1}^{n}\omega_j f(\xi_j)\ \bigg{\rvert} \bigg{]},
\end{equation*}
where $\omega=\{\omega_j\}_{j=1}^n$ are i.i.d Rademacher random variables.
\end{definition}

Then we have the following PAC-type generalization bound \cite[Theorem 3.1]{Mohri:2018}.
\begin{lemma}\label{lem:PAC}
Let $X_1,\ldots,X_n$ be a set of i.i.d. random variables, and let $\mathcal{F}$
be a function class defined on $D$ such that $\sup_{f\in\mathcal{F}}\|f\|_{L^\infty(D)}\leq M_\mathcal{F}<\infty$. Then for any $\tau\in(0,1)$, with probability at least
$1-\tau$:
\begin{equation*}
  \sup_{f\in \mathcal{F}}\bigg|n^{-1}\sum_{j=1}^n f(X_j)-\mathbb{E}[f(X)]\bigg| \leq 2\mathfrak{R}_n(\mathcal{F}) + 2M_\mathcal{F}\sqrt{\frac{\log\frac{1}{\tau}}{2n}}.
\end{equation*}
\end{lemma}

To apply Lemma \ref{lem:PAC}, we bound Rademacher complexities of the function classes $\mathcal{H}_r$, $\mathcal{H}_d$ and $\mathcal{H}_n$.
This follows from Dudley's formula in Lemma \ref{lem:Dudley}.
The next lemma gives useful boundedness and Lipschitz continuity of the DNN function class in terms of $\theta$; see \cite[Lemma 3.4 and Remark 3.3]{JinLiLu:2022} and \cite[Lemma 5.3]{Jin:DNN-Control}.
The estimates also hold when $L^\infty(\Omega)$ is replaced with $L^\infty(\Gamma_D)$ or $L^\infty(\Gamma_N)$. The notation $L^\infty(\Omega;\mathbb{R}^d)$ denotes the $L^\infty(\Omega)$ norm for $\mathbb{R}^d$-valued functions. Note that $\|v_\theta\|_{L^\infty(\Omega)}$ is independent of the depth $L$ since the activation function $ \varrho$ satistifies $\| \varrho\|_{L^\infty(\mathbb{R})}\leq 1$.
\begin{lemma}\label{lem:NN-Lip}
Let $L$, $W$ and $B_\theta$ be the depth, width and maximum weight bound of a DNN function class $\mathcal{W}$, with $N_\theta$ nonzero weights. Then for any $v_\theta\in\mathcal{W}$, the following estimates hold
\begin{enumerate}
\item[{\rm(i)}] $\|v_\theta\|_{L^\infty(\Omega)}\leq WB_\theta$,   $\|v_\theta-v_{\tilde{\theta}}\|_{L^\infty(\Omega)}\leq 2LW^LB_\theta^{L-1}|\theta-\tilde\theta|_{\ell^\infty}$;
\item[{\rm(ii)}] $\|\nabla v_\theta\|_{L^\infty(\Omega; \mathbb{R}^d)}\leq \sqrt{2}W^{L-1}B_\theta^L$,
$\|\nabla (v_\theta-v_{\tilde{\theta}})\|_{L^\infty(\Omega; \mathbb{R}^d)}\leq  \sqrt{2}L^2W^{2L-2}B_\theta^{2L-2}|\theta-\tilde\theta|_{\ell^\infty}$;
\item[{\rm(iii)}] $\|\Delta v_\theta\|_{L^\infty(\Omega)}\leq 2LW^{2L-2} B_\theta^{2L} $,
$\|\Delta (v_\theta-v_{\tilde{\theta}})\|_{L^\infty(\Omega)}\leq  8N_\theta L^2W^{3L-3}B_\theta^{3L-3}|\theta-\tilde\theta|_{\ell^\infty}$.
\end{enumerate}
\end{lemma}

Lemma \ref{lem:NN-Lip} implies the uniform boundedness and Lipschitz continuity of functions in $\mathcal{H}_r$, $\mathcal{H}_d$ and $\mathcal{H}_n$.
\begin{lemma}\label{lem:fcn-Lip}
There exists a constant $c$ depending on $\| f \|_{L^\infty(\Omega)}$, $\|\Delta(\eta_\rho s)\|_{L^\infty(\Omega)}$ and $B_\gamma$ such that
\begin{align*}
  \|h( w_\theta,\gamma)\|_{L^\infty(\Omega)} &\le c L^2 W^{4L-4}B_\theta^{4L}, \quad \forall h \in \mathcal{H}_r,\\
  \|h(w_\theta )\|_{L^\infty(\Gamma_D)} &\le c W^2B_\theta^2,\quad \forall  h \in \mathcal{H}_d,\\
   \|h( w_\theta)\|_{L^\infty(\Gamma_N)} &\le cW^{2L-2}B_\theta^{2L},\quad  \forall h \in \mathcal{H}_n.
\end{align*}
Moreover, the following Lipschitz continuity estimates hold
\begin{align*}
\|h(w_\theta,\gamma) - \tilde h(w_{\tilde\theta},\title{\gamma}) \|_{L^\infty(\Omega)} &\le cN_\theta L^3W^{5L-5}B_\theta^{5L-3}(| \theta - \tilde\theta|_{\ell^\infty}+| \gamma-\tilde \gamma|),\quad \forall  h,\tilde h \in \mathcal{H}_r,\\
\|h( w_\theta )- h(w_{\tilde\theta})\|_{L^\infty(\Gamma_D)} &\le cLW^{L+1}B_\theta^L| \theta - \tilde \theta  |_{\ell^\infty},\quad  \forall  h,\tilde h \in \mathcal{H}_d,\\
  \|h( w_\theta)-\tilde h(w_{\tilde\theta})\|_{L^\infty(\Gamma_N)} &\le c L^2 W^{3L-3} B_\theta^{3L-2}| \theta - \tilde \theta  |_{\ell^\infty}, \quad \forall h,\tilde h \in \mathcal{H}_n.
\end{align*} \end{lemma}
\begin{proof}
All the estimates are direct from Lemma \ref{lem:NN-Lip}.
\end{proof}

Next, we state Dudley's lemma (\cite[Theorem 9]{lu2021priori}, \cite[Theorem 1.19]{wolf2018mathematical}), which bounds Rademacher complexities using covering number.
\begin{definition}
Let $(\mathcal{M},m)$ be a metric space of real valued functions, and $\mathcal{G}\subset \mathcal{M}$. A set $\{x_i\}_{i=1}^n \subset \mathcal{G}$ is called an $\epsilon$-cover of $\mathcal{G}$ if for any $x\in \mathcal{G}$, there exists an element $x_i\in \{x_i\}_{i=1}^n$ such that $m(x, x_i) \leq \epsilon$. The
$\epsilon$-covering number $\mathcal{C}(\mathcal{G}, m, \epsilon)$ is the minimum
cardinality among all $\epsilon$-covers of $\mathcal{G}$ with respect to $m$.
\end{definition}

\begin{lemma}\label{lem:Dudley}
Let $M_\mathcal{F}:=\sup_{f\in\mathcal{F}} \|f\|_{L^{\infty}(\Omega)}$, and $\mathcal{C}(\mathcal{F},\|\cdot\|_{L^{\infty}(\Omega)},\epsilon)$ be the covering number of $\mathcal{F}$. Then the Rademacher complexity $\mathfrak{R}_n(\mathcal{F})$ is bounded by
\begin{equation*}
\mathfrak{R}_n(\mathcal{F})\leq\inf_{0<s< M_\mathcal{F}}\bigg(4s\ +\ 12n^{-\frac12}\int^{M_\mathcal{F}}_{s}\big(\log\mathcal{C}(\mathcal{F},\|\cdot\|_{L^{\infty}(\Omega)},\epsilon)\big)^{\frac12}\ {\rm d}\epsilon\bigg).
\end{equation*}
\end{lemma}

\section{Modified Helmholtz equation}
In this supplement, we extend the proposed SEPINN to the modified Helmholtz equation equipped with mixed boundary conditions:
\begin{align}\label{helmholtz}
	\left\{\begin{aligned}
		-\Delta u +A^2u&= f,&& \mbox{in }\Omega,\\
		u&=0,&& \mbox{on }\Gamma_D,\\
		{\partial_n u}&=0,&& \mbox{on }\Gamma_N,
	\end{aligned}\right.
\end{align}
where $A>0$ is the wave number. 
It is also known as the screened Poisson problem in the literature.
\subsection{Two-dimensional case}
Following the discussion in Section \ref{2d domain}, we use the basis $\{\phi_{j,k}\}_{k=1}^\infty$ to expand the solution $u$, find the singular term and split it from the solution $u$ of problem \eqref{helmholtz}. Since $u$ cannot be expressed in terms of an elementary function, the derivation of the singularity splitting differs slightly from that for the Poisson equation. Nonetheless, a similar decomposition holds; see \cite[Theorem 3.3]{2011Singular} for the detailed proof.
\begin{theorem}
	Let $\boldsymbol{v}_j,j=1,\cdots,M$, be the vertices of $\Omega$ whose interior angles $\omega_j,j=1,\cdots,M$, satisfy \eqref{condition}. Then the unique weak solution $u\in H^1(\Omega)$ of \eqref{helmholtz} can be decomposed into 
\begin{equation*}u=w+\sum_{j=1}^{M}\sum_{i\in\mathbb{I}_j}\gamma_{j,i}\eta_{\rho_j}(r_j) {\rm e}^{-Ar_j}s_{j,i}(r_j,\theta_j),\quad \mbox{with } w\in H^2(\Omega). 
\end{equation*}
	Moreover, the following a priori estimate holds
$$|w|_{H^2(\Omega)}+A|w|_{H^1(\Omega)}+A^2\|w\|_{L^2(\Omega)}\leq C\|f\|_{L^2(\Omega)}.$$
\end{theorem}

Once having the singularity solutions $s_{j,i}$, one can learn the parameters $\gamma_{j,i}$ and the DNN parameters $\theta$ of the regular part $ w_\theta $ as Algorithm \ref{algorithm:2d}. Thus, SEPINN applies equally well to the (modified) Helmholtz case. In contrast, the FEM uses an integral formula and numerical integration when calculating the flux intensity factors \cite{StephanWhiteman:1988,2012Flux}, which can be rather cumbersome. Further, such extraction formulas are still unavailable for the Helmholtz equation. In SEPINN, training the coefficients directly along with the DNN not only reduces manual efforts, but also increases the computational efficiency.

\subsection{Three-dimensional case}
Following the discussion in Section \ref{3d domain}, we employ the orthogonal basis $\{Z_n\}_{n=0}^\infty$ in Table \ref{table 2.3} to expand $u$ and $f$, cf. \eqref{fourier u 3d} and \eqref{fourier f 3d}, and substitute them into problem \eqref{helmholtz} leads to
\begin{align}\label{3dhelmholtz:fourier equation}
	\left\{\begin{aligned}
		-\Delta u_n+(\xi_{j,n}^2+A^2)u_n&=f_n,&& \mbox{in }\Omega_0,\\
		u_n&=0,&& \mbox{on }\Gamma_D,\\
		\partial_n u_n&=0,&& \mbox{on }\Gamma_N.
	\end{aligned}\right.
\end{align}
The only difference of \eqref{3dhelmholtz:fourier equation} from \eqref{3d:fourier equation} lies in the wave number. Upon letting  $\widetilde{A}_{j,n}=(\xi_{j,n}^2+A^2)^{1/2}$, we can rewrite Theorems \ref{theorem:3d split} and \ref{theorem: 3dsplit2} with $\xi_{j,n}=\widetilde{A}_{j,n}$, and obtain the following result.
\begin{theorem}
For each $f\in L^2(\Omega)$, let $u\in H^1(\Omega)$ be the unique weak solution to  problem \eqref{helmholtz}. Then $ u $ can be split into a sum of a regular part $ w $ and a singular part $S$ as
\begin{align*} 
u=w+S,\quad \mbox{with }w\in H^2(\Omega), \,\, S=\sum_{j=1}^M\sum_{i\in\mathbb{I}_j}S_{j,i}(x_j,y_j,z),\\
S_{j,i}(x_j,y_j,z)=\widetilde{\Phi}_{j,i}(r_j,z)\eta_{\rho_j}(r_j)s_{j,i}(r_j,\theta_j),
\end{align*}
with 
$$ \widetilde{\Phi}_{j,i}(x_j,y_j,z)=\frac{1}{2}\gamma_{j,i,0}+\sum_{n=1}^\infty\gamma_{j,i,n}{\rm e}^{-\widetilde{A}_{j,n}r_j}Z_n(z).$$
\end{theorem}

Note that the expressions for the coefficients $ \gamma_{j,i,n} $  differ from that in Section \ref{3d domain}, which however are not needed for implementing SEPINN and thus not further discussed. Nonetheless, we can still use both SEPINN-C and SEPINN-N to solve 3D Helmholtz problems.

\subsection{Numerical experiment for the modified Helmholtz equation}
The next example illustrates SEPINN on the modified Helmholtz equation in two- and three-dimensional cases.
\begin{example}\label{exam:helmholtz}
Consider the following problem in both 2D and 3D domains:
\begin{equation*}
	-\Delta u + \pi^2u = f, \quad\mbox{in}\quad\Omega
\end{equation*}
with the source $f$ taken to be
$f=(4x^3+6x)y{\rm e}^{x^2-1}({\rm e}^{y^2-1}-1)-(4y^3+6y)x{\rm e}^{y^2-1}({\rm e}^{x^2-1}-1)+\pi^2xy({\rm e}^{x^2-1}-1)({\rm e}^{y^2-1}-1)-\Delta({\rm e}^{-\pi r}\eta_\rho s)+\pi^2{\rm e}^{-\pi r}\eta_\rho s$ for the 2D domain $\Omega=\Omega_1=(-1,1)^2\backslash([0,1)\times(-1,0])$, and $f	=-((4x^3+6x)y{\rm e}^{x^2-1}({\rm e}^{y^2-1}-1)-(4y^3+6y)x{\rm e}^{y^2-1}({\rm e}^{x^2-1}-1)
+2\pi^2xy({\rm e}^{x^2-1}-1)({\rm e}^{y^2-1}-1))\sin(\pi z)-\Delta({\rm e}^{-\sqrt{2}\pi r}\eta_\rho s\sin(\pi z))+\pi^2{\rm e}^{-\sqrt{2}\pi r}\eta_\rho s\sin(\pi z)$ for the 3D domain $\Omega=\Omega_2=\Omega_1\times(-1,1)$,
with $\rho=1$ and $R=\frac{1}{2}$ in \eqref{cutoff}, and $s=r^{\frac{2}{3}} \sin(\frac{2 \theta}{3})$, $\Gamma_D=\partial\Omega$. The exact solution $u$ is given by
	\begin{align}
		u=\left\{\begin{aligned}
			&xy({\rm e}^{x^2-1}-1)({\rm e}^{y^2-1}-1)+ {\rm e}^{-\pi r}\eta_\rho s,\quad\Omega=\Omega_1,\\
			&xy({\rm e}^{x^2-1}-1)({\rm e}^{y^2-1}-1)\sin(\pi z)+ {\rm e}^{-\sqrt{2}\pi r}\eta_\rho s\sin(\pi z),\quad\Omega=\Omega_2.
		\end{aligned}\right.
	\end{align}
\end{example}

For the 2D problem, we employ a 2-20-20-20-1 DNN. The first stage of the PF strategy yields an estimate $\widehat{\gamma}^*=0.9934$, and the final prediction error $e$ is  7.15e-3. This accuracy is comparable with that for the Poisson problem in Example \ref{exam:2d-lshape}.
In the 3D case, we have the coefficients $\gamma_1^*=1$ and $\gamma_n^*=0$ for $n\in\mathbb{N}\backslash\{1\}$. In SEPINN-C, we take $N=10$ terms. 
In the first stage, we set the learning rate 1.0e-3 for the DNN parameters $\theta$ and 4.0e-3 for coefficients $\boldsymbol\gamma$.  The estimated $\widehat{\boldsymbol\gamma}^*$ are shown in Table \ref{table:para3}, which approximate accurately the exact coefficients, and the prediction error $e$ after the second stage is 4.34e-2.
In SEPINN-N, we use five 4-layer (3-20-20-20-1) DNNs to approximate $w$ and 3-10-10-10-1 to approximate $\Phi$, and optimize the loss with L-BFGS with a maximum of 5000 iterations. The final prediction error $e$ is 3.07e-2. The results are shown in Fig. \ref{fig:exam:helmholtz}. SEPINN-C and SEPINN-N give very similar pointwise errors, and the singularity at the reentrant corner is accurately resolved, due to singularity enrichment. This shows clearly the flexibility of SEPINN for the screened Poisson equation with geometric singularities.

\begin{table}[H]
	\centering\begin{threeparttable}
		
		\caption{\label{table:para3} The estimated coefficients $\gamma_n$ in the 3D case for Example \ref{exam:helmholtz}, with five significant digits.}
		\begin{tabular}{c|c|c|c|c|c|c}
			\toprule
			& $\gamma_0$ & $\gamma_1$ & $\gamma_2$ & $\gamma_3$ & $\gamma_4$ & $\gamma_5$ \\
			\hline
			estimate & 6.2415{\rm e}-5 & 9.9140{\rm e}-1 & 3.1683{\rm e}-3 & -1.5563{\rm e}-3 & 1.3309{\rm e}-3 & -5.9875{\rm e}-4\\
			\midrule
			& $\gamma_6$ & $\gamma_7$ & $\gamma_8$ & $\gamma_9$ & $\gamma_{10}$ & \\
			\hline
			estimate & 2.0147{\rm e}-3 & 1.2044{\rm e}-3 & 3.6573{\rm e}-4 & -3.8277{\rm e}-4 & -2.4552{\rm e}-3 & \\
			\bottomrule
		\end{tabular}
	\end{threeparttable}
\end{table}

\renewcommand{\figheight}{3.5cm}
\begin{figure}[H]
	\centering\setlength{\tabcolsep}{0pt}
	\begin{tabular}{ccc}
		\includegraphics[height=\figheight]  {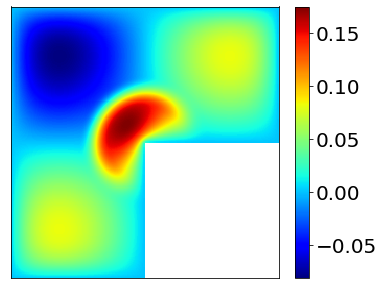} & \includegraphics[height=\figheight]  {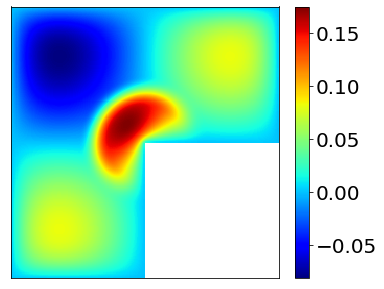} & \includegraphics[height=\figheight]  {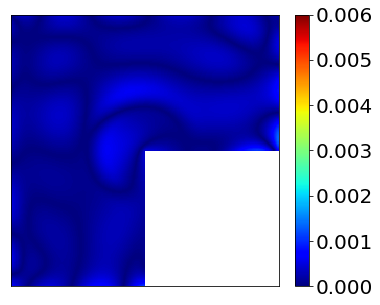}\\
		\includegraphics[height=\figheight]  {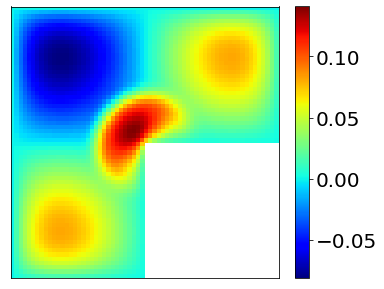} & \includegraphics[height=\figheight]  {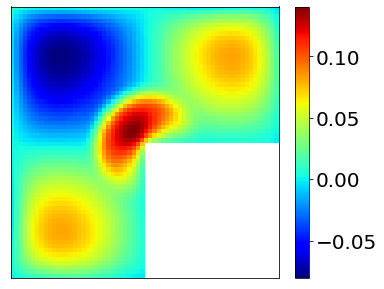} & \includegraphics[height=\figheight]  {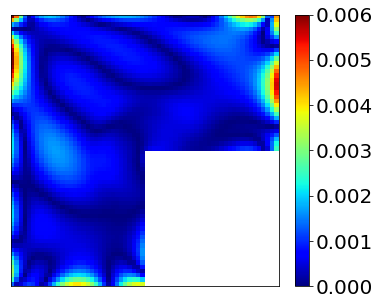} \\
		\includegraphics[height=\figheight]  {true-5-cut.png} &
		 \includegraphics[height=\figheight]  {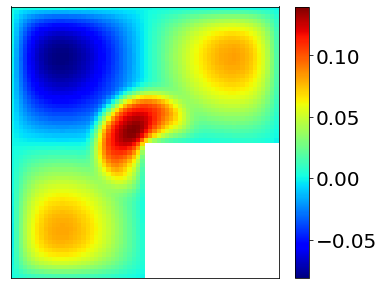} & \includegraphics[height=\figheight]  {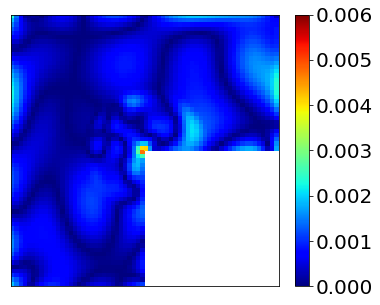}\\
		(a) exact & (b) predicted  & (c) error \\
	\end{tabular}
	\caption{\label{fig:exam:helmholtz} 2D problem with SEPINN {\rm(}top{\rm)} and 3D problem with SEPINN-C {\rm(}middle{\rm)} and SEPINN-N {\rm(}bottom{\rm)} for Example \ref{exam:helmholtz}, with slice at $z=\frac{1}{2}$.}
\end{figure}

\bibliographystyle{siam}
\bibliography{refer}
\end{document}